\documentclass[aos]{imsart}

\RequirePackage{amsthm,amsfonts,amssymb}
\RequirePackage[numbers,sort]{natbib}

\RequirePackage[colorlinks,citecolor=blue,urlcolor=blue]{hyperref}
\RequirePackage{graphicx}

\usepackage[reqno]{amsmath}
\usepackage{rotating}
\usepackage{enumerate}% allows enumerate with latin characters
\usepackage{dsfont}% allows doublestroke symbols with \mathds{}
\usepackage{ulem}% allows strikeout , corss-out, etc.
\usepackage[usenames,dvipsnames]{color}
\usepackage{mathrsfs}
\usepackage{placeins}
\usepackage{graphicx}
\usepackage{arydshln}

\usepackage{tikz}
\usetikzlibrary{arrows}
\usepackage{url}
\usepackage{comment}
\usepackage{multicol}
\usepackage{cleveref}
\usepackage{stmaryrd}
\usepackage{booktabs}
\usepackage{pdflscape}

\usepackage{natbib}

\numberwithin{table}{section}
\numberwithin{equation}{section}

\theoremstyle{plain} %default (text in italic)
\newtheorem{theorem}{Theorem}[section]
\newtheorem{lemma}[theorem]{Lemma}
\newtheorem{proposition}[theorem]{Proposition}
\newtheorem{corollary}[theorem]{Corollary}
\newtheorem{assumptionletter}{{{Assumption}}}

\theoremstyle{remark}
\newtheorem{definition}[theorem]{Definition}

\newtheorem{example}[theorem]{Example}

\newtheorem{remark}[theorem]{Remark}

\def\tablename{\footnotesize Table}

\newcommand{\bthe}{\begin{theorem}}
\newcommand{\ethe}{\end{theorem}}

\newcommand{\ben}{\begin{enumerate}}
\newcommand{\een}{\end{enumerate}}

\newcommand{\bit}{\begin{itemize}}
\newcommand{\eit}{\end{itemize}}

\newcommand{\beq}{\begin{equation}}
\newcommand{\eeq}{\end{equation}}

\newcommand{\ble}{\begin{lemma}}
\newcommand{\ele}{\end{lemma}}

\newcommand{\bde}{\begin{definition}\rm}
\newcommand{\ede}{\halmos\end{definition}}

\newcommand{\bco}{\begin{corollary}}
\newcommand{\eco}{\end{corollary}}

\newcommand{\bpr}{\begin{proposition}}
\newcommand{\epr}{\end{proposition}}

\newcommand{\brem}{\begin{remark}\rm}
\newcommand{\erem}{\end{remark}}

\newcommand{\bproof}{\begin{proof}}
\newcommand{\eproof}{\end{proof}}

\newcommand{\bexam}{\begin{example}\rm}
\newcommand{\eexam}{\end{example}}

\newcommand{\bfi}{\begin{fig}}
\newcommand{\efi}{\end{fig}}

\newcommand{\btab}{\begin{tab}}
\newcommand{\etab}{\end{tab}}

\newcommand{\beao}{\begin{eqnarray*}}
\newcommand{\eeao}{\end{eqnarray*}\noindent}

\newcommand{\balo}{\begin{align*}}
\newcommand{\ealo}{\end{align*}}

\newcommand{\balm}{\begin{align}}
\newcommand{\ealm}{\end{align}\noindent}

\newcommand{\beam}{\begin{eqnarray}}
\newcommand{\eeam}{\end{eqnarray}\noindent}

\newcommand{\barr}{\begin{array}}
\newcommand{\earr}{\end{array}}

\newcommand{\E}{\mathbb{E}}
\newcommand{\M}{\mathbb{M}}
\newcommand{\N}{\mathbb{N}}
\renewcommand\P{\mathbb{P}}
\newcommand{\R}{\mathbb{R}}

\newcommand{\bbO}{\mathbb{E}}

\def\MRV{\mathcal{MRV}}

\def\RV{\mathcal{RV}}

\def\CA{\mathbb{CA}}

\def\bI{\mathbb{I}}

\def\cB{\mathcal{B}}

\def\e{\text{e}}

\def\bzero{\boldsymbol 0}
\def\bone{\boldsymbol 1}
\def\bA{\boldsymbol A}
\def\bB{\boldsymbol B}
\def\bD{\boldsymbol D}
\def\bR{\boldsymbol R}

\def\bZ{\boldsymbol Z}
\def\bL{\boldsymbol L}
\def\bQ{\boldsymbol Q}
\def\bp{\boldsymbol p}
\def\bq{\boldsymbol q}

\def\bx{\boldsymbol x}
\def\by{\boldsymbol y}
\def\bz{\boldsymbol z}

\def\ba{\boldsymbol a}

\def\bc{\boldsymbol c}

\DeclareMathOperator*{\argmin}{arg\,min}

\newcommand{\vague}{\stackrel{\lower0.2ex\hbox{$\scriptscriptstyle
                    \it{v} $}}{\rightarrow}}
\newcommand{\weak}{\stackrel{\lower0.2ex\hbox{$\scriptscriptstyle
                    \it{w} $}}{\rightarrow}}
\newcommand{\what}{\stackrel{\lower0.2ex\hbox{$\scriptscriptstyle
                    \it{\hat{w}} $}}{\rightarrow}}
\newcommand{\eqdis}{\stackrel{\lower0.2ex\hbox{$\scriptscriptstyle
                    \mathrm{d}$}}{=}}
\newcommand{\distr}{\stackrel{\lower0.2ex\hbox{$\scriptscriptstyle
                    \it{d} $}}{\rightarrow}}
\allowdisplaybreaks %% display breaks for align
\definecolor{darkgreen}{RGB}{0,139,0}

\begin{document}

\begin{frontmatter}
%\title{Multivariate Heavy-Tailed Ruin I:\\ Geometry and Hidden Tail Layers at Finite Horizons}
%\title{Ruin for Multivariate heavy-tailed claim processes in finite horizon: \vspace*{0.2cm}\\ Geometry of ruin sets and hidden tail dependence}
\title{Multivariate Ruin in L\'evy-driven \\  Insurance Risk Models under Heavy Tails}
%\title{Finite-Time Ruin for Multivariate \vspace*{0.2cm} \\ Insurance Risk Processes with Heavy-Tailed Claims}
\runtitle{Multivariate Ruin under Heavy Tails}

\begin{aug}
  \author[A]{\fnms{Bikramjit} \snm{Das}\ead[label=e1]{bikram@sutd.edu.sg}\orcid{0000-0002-6172-8228}}
  \and
  \author[B]{\fnms{Vicky} \snm{Fasen-Hartmann}\ead[label=e2]{vicky.fasen@kit.edu}\orcid{0000-0002-5758-1999}}
  \address[A]{Engineering Systems and Design, Singapore University of Technology and Design\printead[presep={,\ }]{e1}}
  \address[B]{Institute of Stochastics, Karlsruhe Institute of Technology\printead[presep={,\ }]{e2}}
  \runauthor{B. Das and V. Fasen-Hartmann}
\end{aug}

\begin{abstract}
The paper studies finite-horizon ruin for an insurance company with several lines of business admitting heavy-tailed claims. The claim amount process of the insurance company is modeled as a multivariate increasing Lévy process possibly perturbed with fluctuations. A relevant insolvency event may involve the failure of a single line, of several lines at once, or of groups of lines linked through internal transfers, reinsurance, or a guarantee arrangement. These events have different geometries and need not occur on the same probability scale. Under asymptotic tail independence among claims from different lines, classical regular variation assigns zero limiting probability to a multi-line failure and is therefore unable to identify either the ruin probability or the corresponding solvency-capital requirement. Using a multivariate L\'{e}vy process to model cumulative claims and regular variation on a nested sequence of subcones of the positive orthant to model dependence across lines, 
we provide the exact scaling rate of the ruin probability and the limit behavior. 
Additionally, explicit computations are presented for examples with independent business lines, common arrivals of independent claims, Gaussian copula dependent claims, Marshall--Olkin dependent claims, and bipartite insurance networks. 
Finally, we use Monte Carlo simulations to support our theoretical findings.

\end{abstract}
\begin{keyword}[class=MSC]
\kwd[Primary ]{91G05}
\kwd{60G70}
\kwd[; secondary ]{91B05}
\kwd{60G51}
%\kwd{62H05}
\end{keyword}

\begin{keyword}
\kwd{asymptotic tail independence}
\kwd{capital allocation}
\kwd{heavy-tailed claims}
\kwd{L\'evy processes}
\kwd{multivariate regular variation}
\kwd{ruin probability}
\kwd{solvency capital}
\end{keyword}
\end{frontmatter}

\section{Introduction}\label{sec:intro}
An insurance company writing several lines of business may become insolvent in a variety of ways. Accordingly, the company may monitor the failure of at least one line, the insolvency of a prescribed number of lines, the simultaneous failure of specified subsidiaries, the breach of group-level reinsurance constraints, or a shortfall in a common guarantee fund. The probability of such an event depends on the marginal tail distribution of the claims, extremal dependence across lines, and the geometry of the insolvency region. Thus, traditional models calibrated to single-line ruin, or at-least-one-line ruin, may give an incorrect decay rate and hence a trivial probability approximation. % when the relevant event involves several lines, groups, or aggregate constraints.

We consider an insurance company with $d\geq2$ business lines indexed by $\mathbb I_d:=\{1,\ldots,d\}$ and model the insurance claim process by an increasing $\R^d_+$-valued L\'evy process, also called a \textit{subordinator}. Recall that a multivariate L\'evy process $\bL=(\bL(t))_{t\geq0}$ is a stochastic process with $\bL(0)=\bzero$ $\mathbb{P}$-almost surely, stationary and independent increments, and c\`adl\`ag sample paths (cf. \citet{sato:1999}). 
The L\'evy measure $\Pi(B)$ describes the expected number of jumps of the L\'evy process in $[0,1]$ that lie in the set $B$ (cf. \Cref{sec:subordinator-rv}). 
A special Lévy process is a multivariate compound Poisson process in which claims arrive at the jump times of a Poisson process, and the i.i.d. claim sizes are independent of that process (cf. \Cref{example:compounsPoisson}).  
 
 The company starts with the total initial capital $u>0$ and using the allocation vector \linebreak $\ba=(a_1,\ldots,a_d)^\top\in(0,\infty)^d$, satisfying $\sum_{j=1}^da_j=1$, assigns capital $ua_j$ to line $j$. In addition, the insurance company has premium income. The premium rates of the different business lines are given by the vector $\bp=(p_1,\ldots,p_d)^\top\in\R_+^d$. Then the \textit{insurance risk process (reserve process)} is
\begin{equation*} 
   \bR_u(t)=u\ba+t\bp-\bL(t),\qquad t\geq0.
\end{equation*}
A well-known special case is the \textit{Cram\'er-Lundberg model} where $d=1$ and $\bL$ is a compound Poisson process (cf. \Cref{example:compounsPoisson}).
Now, ruin occurs when the risk process enters a \textit{ruin set}, or \textit{insolvency region}, $\Lambda\subset\R^d$. For a finite horizon $T>0$, the ruin probability is
\begin{equation*} 
 \begin{aligned}
 \psi_u(\Lambda,T)
 &=\P\!\left(
 \bR_u(t)\in\Lambda
 \text{ for some }t\in[0,T]
 \right).
 \end{aligned}
\end{equation*}

The basic heavy-tailed ruin model is the {Cram\'er--Lundberg model} with subexponential (and regularly varying) claim size distributions.  
 The literature \cite{embrechts:kluppelberg:mikosch:1997,rolski:schmidli:schmidt:teugels:1999,Asmussen:Albrecher:2010, Konstantinides:2017}  shows that, for heavy-tailed claims, ruin is typically caused by a single exceptionally large claim or jump, and the results therein provide the benchmark tail-equivalence arguments behind many finite- and infinite-horizon approximations.

Multivariate ruin is more delicate because both claim dependence and the shape of the insolvency event affect the ruin behavior. 
Multivariate regular variation and sample-path large deviations provide the probability foundations for heavy-tailed multivariate ruin asymptotics for Lévy-driven insurance-risk models \cite{hult:lindskog:2005SPA,hult:lindskog:mikosch:samorodnitsky:2005,Hult:Lindskog:2011, rhee:blanchet:zwart:2019}. Related multivariate subexponential distributions (which contain regularly varying distributions) along with renewal, network, and investment-risk models are treated, among others, in \cite{li:liu:tang:2007,shen:zhang:2013,Behme2020ruin,konstantinides:2025,samorodnitsky:sun:2016,yang:li:2014,konstantinides:li:2016,li2015asymptotic,cheng2024multivariate}. All of this literature is driven by the one-big-jump principle. 

The objective of this paper is to identify the rate of decrease of the ruin probability $\psi_u(\Lambda,T)$ as $u\to\infty$, along with the limiting value under this scaling for heavy-tailed claim processes $\bL$. 
The challenge in this multivariate setting arises under \textit{asymptotic tail independence} of the claim sizes 
 across different business lines (cf. \cite{das:fasen:2026}) or, more generally, of the underlying Lévy measure $\Pi$ of $\bL$, in which case the limit measure of standard multivariate regular variation is concentrated on the axes. This is prevalent in many popular models. The problem is that, under asymptotic tail independence, the standard multivariate regular variation correctly describes the ruin probability of a single line, but assigns zero mass to events requiring two or more lines to become insolvent. Although mathematically valid, such a zero limit gives neither the decay rate of the ruin probability nor an approximation suitable for determining solvency capital. 

The key contribution of this paper is the calculation of the exact ruin probability rate using the geometry of the insolvency event and multivariate regular variation of Lévy processes on nested subcones of the positive orthant. 
In earlier work, \cite{das:fasen:2027} related multivariate regular variation on these subcones of the L\'evy measure to L\'evy processes, and extended the one-big-jump asymptotics to few-big-jump and multiple-coordinate threshold crossings. 
 This provides the key ingredient for establishing the limiting result of the ruin probability under the correct scaling rate in the present paper. Thus, we are able to characterize the precise asymptotic ruin behavior for a wide variety of insolvency  sets, including the failure  of groups of lines linked through internal transfers,
reinsurance, or a guarantee arrangement that extends beyond just single-line ruin. 

Additionally, our risk model is quite general. It captures not only compound Poisson claim processes but also more general increasing Lévy processes as claim amount processes.
The insurance risk process can be augmented by a perturbation process $\bQ=(\bQ(t))_{t\geq 0}$ to account for additional uncertainty in claim payments or premium income, which can, for example, be a Brownian motion. This idea dates back to \cite{Gerber1970}, who introduced a Brownian motion as a diffusion component into the classical Cram\'er--Lundberg model, and has since been studied and extended by numerous researchers.
The process $\bQ=(\bQ(t))_{t\geq 0}$ is supposed to be an $\R^d$-valued process independent of $\bL$ whose tail is negligible relative to the heavy-tailed claims. The \textit {perturbed insurance risk process} is then
\begin{equation*}%\label{eq:perturbed-reserve}
 \bR_u^{\bQ}(t)=u\ba+t\bp-\bL(t)+\bQ(t),
 \qquad t\geq 0.
\end{equation*}
All of our results are developed for this general model.

The paper is organized as follows. 
In Section~\ref{sec:prelim}, we introduce multivariate regular variation on subcones of $\R_+^d$ and relate it to multivariate regular variation of Lévy measures and increasing Lévy processes on these subcones. These findings motivate the fundamental Assumption \ref{ass:riskmodel} of this paper, stated at the end of this section. \Cref{sec:main} presents the main results concerning the decay rate and limiting ruin probability of the perturbed insurance risk process, as well as the underlying risk process. The analysis is based on the geometric properties of the insolvency set. Furthermore, we discuss an application to solvency capital determination and optimal capital allocation. Next, in \Cref{ex:ruin-geometries}, we consider several insolvency sets arising from capital transfers, grouped reinsurance and guarantee funds.  Dependence-specific constants for independent lines, common arrivals with independent claims, Gaussian-copula dependent claims, Marshall--Olkin dependent claims, and bipartite network models are computed in Section~\ref{sec:examples}. 
Numerical simulations supporting the theoretical results are provided in Section~\ref{sec:simulation}, and concluding remarks are given in Section~\ref{sec:conclusion}.
Finally, \Cref{tab:rate-summary} in \Cref{app:summary-rates} summarizes the behavior of the ruin probability across the different models and insolvency sets considered in this paper.

\subsection*{Notation}\label{sec:notation}

Throughout the paper, $d\geq2$ is the number of business lines,
$\mathbb I_d:=\{1,\ldots,d\}$ is their index set, and
$\R_+:=[0,\infty)$. Vectors are written in boldface; for $\bz\in\R^d$, the superscript $\top$ in $\bz^{\top}$ denotes the transpose. Inequalities between vectors are interpreted coordinatewise. The symbols $\bzero$ and $\bone$ denote vectors of zeros and ones of the required dimension. The inner product and the sup norm are denoted by $\langle\cdot,\cdot\rangle$ and $\|\cdot\|_\infty$, respectively.
For $\bz\in\R^d$, the notation $z_{(1)}\geq\cdots\geq z_{(d)}$ denotes the decreasing order statistics of its coordinates. We write $|S|$ for the cardinality of a finite set $S$ and $\mathds 1_B$ for the indicator of an event $B$. For a set $B$, $\partial B$ denotes its boundary. If $\mu$ is a measure, then $B$ is a $\mu$-continuity set when $\mu(\partial B)=0$. The notation $f(u)\sim g(u)$ means $f(u)/g(u)\to1$ as $u\to\infty$.
For a nondecreasing function $g:\R_+\to\R_+$, its generalized inverse is
$
 g^{\leftarrow}(u):=\inf\{t>0:g(t)\geq u\}.
$
Finally, Table~\ref{tab:notation} gives an overview of the principal model and tail quantities used throughout the paper. 
\begin{table}[t]
\caption{Principal notation and its interpretation.}\label{tab:notation}
\centering
\begin{tabular}{ll}
\toprule
Symbol & Meaning\\
\midrule
$u$ & total initial capital\\
$\ba=(a_1,\ldots,a_d)^\top$ & positive capital-allocation vector; $\sum_{j=1}^d a_j=1$\\
$\bp=(p_1,\ldots,p_d)^\top$ & nonnegative premium-rate vector\\
$\bL(t)$ & cumulative-claim vector at time $t$\\
$\bR_u(t),\ \bR_u^{\bQ}(t)$ & unperturbed and perturbed risk processes at time $t$\\
$\bQ(t)$ & perturbation process at time $t$\\
$\Lambda$ & multivariate ruin or insolvency region\\
$\psi_u(\Lambda,T)$ & probability that $\bR_u$ enters $\Lambda$ by time $T$\\
$\psi_u^{\bQ}(\Lambda,T)$ & probability that $\bR_u^{\bQ}$ enters $\Lambda$ by time $T$\\
$B_{\Lambda}$ & normalized claim region corresponding to $\Lambda$\\
$i$ & subcone level (at least $i$ coordinates are extreme)\\
$\alpha_i$ & tail index on the $i$th subcone\\
$b_i,\ b_i^{\leftarrow}$ & scaling function on the $i$th subcone and its generalized inverse\\
$\mu_i$ & limit measure on the $i$th subcone\\
$f_i(T)$ & factor describing the relevant large-jump mechanism over $[0,T]$\\
\bottomrule
\end{tabular}
\end{table}

\newpage

%\section{Model for multivariate insurance claims}

\section{Multivariate regular variation on subcones}\label{Prelim:MRV} \label{sec:prelim}

In order to derive the scaling rate and the asymptotic behavior of the ruin probability, we employ the theory of multivariate regular variation on different subcones for the Lévy measure $\Pi$ of the insurance claim process $\bL$ in $\R_+^d$, where $\bL$ is a Lévy process. We next introduce the required concepts and notation.

The subcones of $\R_+^d$ used in this paper are defined for $i\in\mathbb I_d$ as
\begin{equation*} %\label{eq:subcones}
\bbO_d^{(i)}
 :=\{\bz\in\R_+^d:z_{(i)}>0\}
 =\R_+^d\setminus\CA_d^{(i-1)}
 \qquad \text{ with } \qquad
 \CA_d^{(i-1)}
 :=\{\bz\in\R_+^d:z_{(i)}=0\},
\end{equation*}
where  $z_{(1)}\geq\cdots\geq z_{(d)}$ are the decreasing order statistics of the coordinates $z_1,\ldots,z_d$. Here,  the cone $\bbO_d^{(i)}$ contains vectors with at least $i$ strictly positive coordinates.
As special cases we have $\bbO_d^{(1)}=\R_+^d\setminus\{\bzero\}$ and 
$\bbO_d^{(d)}=(0,\infty)^d$. 
Our approach is to define and then assume that the Lévy measure $\Pi$ of the insurance claim process is multivariate regularly varying on the nested subcones
\[
 \bbO_d^{(1)}\supset\bbO_d^{(2)}\supset\cdots\supset\bbO_d^{(d)}.
 \]
Multivariate regular variation on these subcones provides a sequence of tail indices and limit measures. A ruin region selects the first subcone on which the \textit{normalized claim region} lies, and the limit measure has positive mass (cf. \Cref{sec:main}). This gives both the correct rate and the correct non-zero limiting constant.

\begin{remark}
We denote by $\mathcal{B}(\bbO_d^{(i)})$ the Borel $\sigma$-algebra on $\bbO_d^{(i)}$ and by
 $\mathbb M(\bbO_d^{(i)})$ the space of Borel measures on $\bbO_d^{(i)}$ that are finite on sets bounded away from the cone $\CA_d^{(i-1)}$ endowed with the corresponding $\mathbb M$-convergence topology. The general $\mathbb M$-convergence framework is developed in     \cite{hult:lindskog:2006a,das:mitra:resnick:2013,lindskog:resnick:roy:2014} and its application to aggregation and Lévy processes is studied in
    \cite{das:fasen:2027}.
\end{remark}

We first recall the notion of regular variation for measurable functions, which is needed to define multivariate regular variation. A measurable function $f:\R_+\to\R_+$ is said to be regularly varying at infinity with index $\beta\in\R$, denoted by $f\in\RV_\beta$, if
\[
\lim_{u\to\infty}\frac{f(ux)}{f(u)}=x^\beta,\qquad x>0.
\] 
With this notation in place, we are ready to define multivariate regular variation on the subcones introduced above.

\begin{definition}\label{def:mrv-rv}
\begin{enumerate}
\item[(a)] A Borel measure $\nu$ on $\R_+^d\setminus\{\bzero\}$ is \textit{multivariate regularly varying} on $\bbO_d^{(i)}$ if there exist a regularly varying function $b_i\in\RV_{1/\alpha_i}$ for some $\alpha_i>0$ and a non-zero Borel measure $\mu_i\in\M(\bbO_d^{(i)})$ such that
  \begin{equation*} %\label{eq:mrv-measure}
    \lim_{u\to\infty}u\,\nu\big(b_i(u)\,B\big)=\mu_i(B)
\end{equation*}
for every $\mu_i$-continuity set $B\in\cB(\bbO_d^{(i)})$ that is { bounded away from\/} $\CA_d^{(i-1)}$. We write $\nu\in\MRV(\alpha_i,b_i,\mu_i,\bbO_d^{(i)})$, where some parameters may be dropped for convenience.

\item[(b)] If $\nu$ is a probability measure and $\bZ\sim\nu$, then $\bZ\in\R_+^d$ is multivariate regularly varying on $\bbO_d^{(i)}$ if $\nu\in\MRV(\alpha_i,b_i,\mu_i,\bbO_d^{(i)})$, and we may write $\bZ\in\MRV(\alpha_i,b_i,\mu_i,\bbO_d^{(i)})$.
\end{enumerate}

\end{definition}

The following remarks are helpful and illustrative.

\begin{remark} $\mbox{}$
\begin{itemize}
    \item[(i)] The quantity $1/b_i^{\leftarrow}(u)$ represents the probability
    scale associated with the $i$-th subcone.  If $\bZ=(Z_1,\ldots,Z_d)^{\top}\in\MRV(\alpha_i,b_i,\mu_i,\bbO_d^{(i)})$ and $F_{(i)}$ denotes the distribution
    function of the $i$-th order statistic $Z_{(i)}$ of $Z_1,\ldots,Z_d$, a canonical choice is
    \[
    b_i(u)=F_{(i)}^{\leftarrow}(1-1/u)=\inf\{x\in\R_+:F_{(i)}(x)\ge 1-1/u\}.
    \]
    \item[(ii)] The limit measure $\mu_i$ is homogeneous of order $-\alpha_i$, that is,
    \[
    \mu_i(tB)=t^{-\alpha_i}\mu_i(B),
    \qquad t>0,\; B\in\mathcal{B}(\bbO_d^{(i)}).
    \]

    \item[(iii)] If regular variation holds on all nested subcones
    $
    \bbO_d^{(1)}\supset\bbO_d^{(2)}\supset\cdots\supset\bbO_d^{(d)},
    $
    then necessarily
    $
    \alpha_1\le \alpha_2\le\cdots\le \alpha_d.
    $
    Strict inequalities imply progressively rarer levels of simultaneous extremes. Equality may also occur under asymptotic independence, but is harder to detect statistically.
\end{itemize}
\end{remark}

\subsection{Multivariate regular variation of multivariate subordinators}\label{sec:subordinator-rv}

In the context of this paper, we model the cumulative claim process $\bL$ in $\R_+^d$ as a multivariate regularly varying increasing Lévy process.
 A L\'evy process is characterized by its \textit{L\'evy--Khintchine representation}
$\E(\e^{i\langle\Theta,\bL(t)\rangle})=\exp(-t\,\Psi(\Theta))$
for $\Theta\in\mathbb{R}^d$, where
 \beao
        \Psi(\Theta)=
        -i\langle \gamma,\Theta\rangle
        +\frac{1}{2}\langle\Theta,\Sigma\,\Theta \rangle+\int_{\R^d}
 \left(1-\e^{i\langle\Theta,\bx\rangle}
 +i\langle \bx,\Theta\rangle\mathds 1_{\{\|\bx\|\leq1\}}\right)\,\Pi(d\bx)
    \eeao
with $\gamma\in\mathbb{R}^d$,  a nonnegative-definite
matrix $\Sigma$ in $\mathbb{R}^{d\times d}$ and a Borel measure
 $\Pi$ on $\mathbb{R}^d$, called the \textit{L\'evy measure},
 which satisfies
$\int_{\R^d}\min\{\|\bx\|^2,1\}\,\Pi(d\,\bx)<\infty$
and $\Pi(\bzero)=0$. %Moreover, $\langle \cdot,\cdot\rangle$ denotes the inner product in $\R^d$. 
The L\'evy measure $\Pi(B)$ is the expected number of jumps of the L\'evy process in $[0,1]$ that lie in $ B$.  For our model, unless otherwise stated, we take $\bL$ to be an $\R_+^d$-valued subordinator. Therefore, the Gaussian component of $\bL$ is absent, its paths are coordinatewise non-decreasing, and $\Pi$ is concentrated on $\R_+^d\setminus\{\bzero\}$. A special case is a multivariate compound Poisson process with positive jumps.

\begin{example}[compound Poisson model] \label{example:compounsPoisson}
In a compound Poisson insurance claim model,  the arrival times of the claims are modeled by a Poisson process $N_\lambda=(N_{\lambda}(t))_{t\geq 0}$ with intensity $\lambda>0$. At each arrival time a $d$-dimensional claim $\bZ^{(k)}=(Z_1^{(k)},\ldots,Z_d^{(k)})^\top$ in $\R_+^d$ occurs, i.e., $Z_j^{(k)}$ is the size of the $k$-th claim in business line $j$.  The
sequence of claims $(\bZ^{(k)})_{k\in\N}$ is independent and identically distributed with distribution function $F$, and independent of the Poisson process $N_\lambda$. 
The cumulative claim process can be written as
\[
 \bL(t)=(L_1(t),\ldots,L_d(t))^{\top}=\sum_{k=1}^{N_\lambda(t)}\bZ^{(k)}, \qquad t\geq 0,
\]
and is called a {\textit{compound Poisson process}}. This is a special Lévy process with Lévy measure $\Pi=\lambda F$. Since each individual claim is positive, the component processes $(L_j(t))_{t\geq 0}$ are increasing, $j=1,\ldots,d$; hence such processes  are subordinators. Obviously, the multivariate regular variation of the Lévy measure $\Pi$ on $\bbO_d^{(i)}$
is equivalent to the multivariate regular variation of the claim size distribution $F$ on $\bbO_d^{(i)}$.
\end{example}

\begin{remark}[L\'evy copulas]\label{rem:levy-copula}
For a general multivariate subordinator, the Lévy measure $\Pi$ can be represented through the marginal Lévy measures and a L\'evy copula. Regular variation of the marginals together with an appropriate homogeneous limit of the L\'evy copula yields regular variation of $\Pi$ on the relevant subcones. Pareto L\'evy measures and their relation to multivariate regular variation on $\bbO_d^{(1)}$ are studied in \cite{eder:kluppelberg:2012}.
These results can be extended straightforwardly to $\bbO_d^{(i)}$ (cf. \cite{hua:joe:li:2014} for random vectors). 
\end{remark}

The regular variation of the Lévy measure on a sequence of subcones implies the regular variation of the Lévy process on these subcones, provided that certain assumptions are met. This connection was studied extensively in \citet{das:fasen:2027} and is essential for deriving the asymptotic behavior of ruin probabilities under different insolvency sets. Here, we recall three results from \cite{das:fasen:2027} that provide basic examples for the claim amount process $\bL$.  
In the following discussion, we will assume that the marginal Lévy measures are tail equivalent.
The marginal L\'evy measures, denoted by $\Pi_1,\ldots,\Pi_d$, are tail equivalent if there exist constants $c_j\in(0,\infty)$, with $c_1=1$, such that
\[
 \lim_{u\to\infty}
 \frac{\Pi_j((u,\infty))}{\Pi_1((u,\infty))}=c_j,
 \qquad j=1,\ldots,d.
\]
The first result is \citet[Proposition~5]{das:fasen:2027} and reflects the one-large-multivariate-jump regime.

\begin{proposition}\label{prop:levy-transfer}
Let $\bL$ be an $\R_+^d$-valued L\'evy process with tail-equivalent marginal L\'evy measures. Suppose
$
 \Pi\in\MRV(\alpha_i,b_i,\mu_i,\bbO_d^{(i)}),
 \, i=1,\ldots,d,
$
and
\begin{equation}\label{eq:strict-subadd}
 \alpha_i<\alpha_m+\alpha_{i-m},
 \qquad m=1,\ldots,i-1,
 \quad i=2,\ldots,d.
\end{equation}
Then, for every $T>0$,
\[
 \bL(T)\in\MRV(\alpha_i,b_i,T\mu_i,\bbO_d^{(i)}),
 \qquad i=1,\ldots,d.
\]
Consequently, for every $\mu_i$-continuity set $B\in \mathcal{B}(\bbO_d^{(i)})$ bounded away from $\CA_d^{(i-1)}$,
\[
 \P\!\left(\bL(T)\in uB\right)
 \sim T\Pi(uB),\qquad u\to\infty.
\]
\end{proposition}

\begin{remark}\label{rem:single-multiple-jumps}
The result reflects the classical one-large-jump phenomenon: $i$ components of $\bL(T)$ are jointly large only if they are caused by a single jump of the Lévy process with large values in the corresponding coordinates. Condition \eqref{eq:strict-subadd} rules out the possibility that $i$ large coordinates of $\bL(T)$ result from $m$ large coordinates in one jump and $i-m$ large coordinates in another jump.
\end{remark}

The following two results give benchmark few-large-jump and common-arrival regimes. We begin with \citet[Proposition~6]{das:fasen:2027}.

\begin{proposition}\label{prop:cpp-separate-jumps}
Let $(\bL(t))_{t\geq0}$ be an $\R_+^d$-valued L\'evy process whose L\'evy measure $\Pi$ has tail-equivalent marginal measures with common tail index $\alpha$. Suppose
$
 \Pi\in\MRV(\alpha,b_1,\mu_1,\bbO_d^{(1)}).
$
Moreover, assume that, for $i=2,\ldots,d$ and every rectangular set $B$ in $\bbO_d^{(i)}$ of the form $B=\bigcap_{j\in S}\{\bz\in\bbO_d^{(i)}:z_j>x_j\}$, where $S\subseteq\mathbb I_d$, $|S|=i$, and $x_j>0$ for $j\in S$, we have
\begin{align}\label{eq:adaptrv}
    \lim_{u\to\infty}u\Pi(u^{1/(i(\alpha+\gamma))}B)= 0,
\end{align}
for some $\gamma>0$. This means that $\Pi$ converges sufficiently fast to zero on the cones $\bbO_d^{(i)}$, $i\geq2$, so only the first subcone is active for an individual large jump. Then, for every $T>0$, and $i=1,\ldots,d$,
\begin{equation*} %\label{eq:cpp-separate-jumps-mrv}
 \bL(T)\in
 \MRV\!\left(i\alpha,b_i,T^i\mu_i^L,\bbO_d^{(i)}\right)
 \qquad \text{ with } \qquad
 b_i^{\leftarrow}(u)=\left[b_1^{\leftarrow}(u)\right]^i,
\end{equation*}
and for every $S\subset\mathbb{I}_d$ with  $|S|=i$ and $x_j>0$, $j\in S$,
\begin{equation*} %\label{eq:cpp-separate-jumps-measure}
 \mu_i^L\!\left(
   \bigcap_{j\in S}\{\bz\in\bbO_d^{(i)}:z_j>x_j\}
 \right)
 =\prod_{j\in S}
 \mu_1\!\left(\{\bz\in\bbO_d^{(1)}:z_j>x_j\}\right),
\end{equation*}
which uniquely defines $ \mu_i^L$ with
$\mu_i^L(\bbO_d^{(i+1)})=0$. 
Consequently, for every $\mu_i^L$-continuity set $B\in \mathcal{B}(\bbO_d^{(i)})$ bounded away from $\CA_d^{(i-1)}$,
\[
 \P\!\left(\bL(T)\in uB\right)
 \sim T^i\P\!\left(\bL(1)\in uB\right),
 \qquad u\to\infty.
\]
\end{proposition}

\begin{remark}[Adapted regular variation]\label{rem:arv}
Together with regular variation on the first cone, condition  \eqref{eq:adaptrv} is an example of \textit{adapted multivariate regular variation}. This is a generalization of multivariate regular variation on different subcones that allows for multivariate regular variation of aggregates of random vectors; see \citet{das:fasen:2027} for details.
\end{remark}

\begin{example}[Separate large jumps]\label{ex:separate-large-claims}
Suppose that the marginal Lévy processes $L_1,\ldots,L_d$ are independent and that their marginal Lévy measures are tail-equivalent with common tail index $\alpha$.
In this case, the Lévy measure $\Pi$ is concentrated on the coordinate axes, so the assumptions of \Cref{prop:cpp-separate-jumps} are satisfied.
The event that the coordinates of $\bL(T)$ indexed by $S$ with $\vert S\vert=i$ exceed high positive thresholds is generated by $i$ distinct large jumps of the Lévy process, each contributing to exactly one coordinate. This explains both the exponent $i\alpha$ and the factor $T^i$.
The limiting measure is supported on coordinate faces with exactly $i$ positive coordinates, as reflected by the fact that $
\mu_i^L(\bbO_d^{(i+1)}) = 0.$
\end{example}

The final result in this section is \citet[Proposition~7]{das:fasen:2027}, which includes the important case of a compound Poisson process whose marginal jump-size distributions are independent of each other.
\begin{proposition}\label{prop:poisson-moment-transfer}
Let $\bL$ be an $\R_+^d$-valued L\'evy process whose marginal L\'evy measures are tail equivalent. Suppose that, for $i=1,\ldots,d$,
\[
 \Pi\in\MRV(i\alpha,b_i,\mu_i,\bbO_d^{(i)}) \qquad \text{ with } 
 \qquad
 b_i^{\leftarrow}(u)=\left[b_1^{\leftarrow}(u)\right]^i,
\]
where $b_1\in\RV_{1/\alpha}$. Assume that there are constants $\kappa_1,\ldots,\kappa_d>0$ such that for every $S\subseteq\mathbb I_d$ with $|S|=i$ and $x_j>0$, $j\in S$,
\begin{equation*} %\label{eq:poisson-moment-measure}
 \mu_i\!\left(
   \bigcap_{j\in S}\{\bz\in\bbO_d^{(i)}:z_j>x_j\}
 \right)
 =\prod_{j\in S}\kappa_jx_j^{-\alpha},
 \end{equation*}
and $\mu_i(\bbO_d^{(i+1)})=0$. If $N^*=(N^*(t))_{t\geq 0}$ denotes a unit-rate Poisson process, then, for every $T>0$ and $i=1,\ldots,d$,
\begin{equation*} %\label{eq:poisson-moment-mrv}
 \bL(T)\in
 \MRV\!\left(
   i\alpha,b_i,\E\!\left[N^*(T)^i\right]\mu_i,
   \bbO_d^{(i)}
 \right).
\end{equation*}
\end{proposition}
\begin{remark}
Proposition~\ref{prop:poisson-moment-transfer} differs from Proposition~\ref{prop:cpp-separate-jumps} in its time factor $T^i$ versus $\E\left[N^*(T)^i\right]$. The Poisson moment $\E[N^*(T)^i]$ reflects that the exceedance of $i$ coordinates of $\bL(T)$ is distributed among one or several jumps; hence it combines all relevant allocations of large coordinates across the available jumps. 
\end{remark}

The conclusions of \Cref{prop:levy-transfer,prop:poisson-moment-transfer,prop:cpp-separate-jumps} motivate the following assumption, which will be used throughout the paper.

\begin{assumptionletter}[Insurance claim model]\label{ass:riskmodel}
Suppose that the insurance claim process $\bL$ is an $\mathbb R_+^d$-valued Lévy process. For each $i \in \mathbb I_d$, assume that there exist a constant $\alpha_i > 0$, a scaling function $b_i \in \RV_{1/\alpha_i}$, a non-zero measure $\mu_i$ on $\bbO_d^{(i)}$, and a positive measurable function $f_i:(0,\infty)\to(0,\infty)$ satisfying $f_i(1)=1$, such that, for every $T>0$,
\begin{equation*} %\label{eq:risk-assumption}
 \bL(T)\in\MRV\big(\alpha_i,b_i,f_i(T)\mu_i,\bbO_d^{(i)}\big).
\end{equation*}
\end{assumptionletter}

The normalization $f_i(1)=1$ fixes the scale of $\mu_i$ and makes the measure $\mu_i$ identifiable. With this convention, $f_i$ records the one-large-jump and few-large-jumps mechanism: Proposition~\ref{prop:levy-transfer} gives $f_i(T)=T$, Proposition~\ref{prop:cpp-separate-jumps} gives $f_i(T)=T^i$ and Proposition~\ref{prop:poisson-moment-transfer} gives
$f_i(T)=\E[N^*(T)^i]/\E[N^*(1)^i]$, with the factor $\E[N^*(1)^i]$ taken in $\mu_i$. Note that Assumption \ref{ass:riskmodel} allows for jump mechanisms resulting in other functional forms for $f_i$ which have not been explored here.

\section{Finite-time ruin probability}\label{sec:main}
The following section is devoted to deriving the asymptotic behavior of the finite-time ruin probability of an insurance company under an increasing conic ruin region.
In total we have $d\geq 2$ business lines, 
and the component $L_j(t)$ of the L\'evy process $\bL$ represents the cumulative claim   in business line  $j\in\mathbb{I}_d$, by time $t$. Given total initial capital $u>0$, an allocation vector $\ba\in(0,\infty)^d$ with $\sum_{j=1}^da_j=1$,  a premium-rate vector $\bp\in\R_+^d$, and an  $\R^d$-valued fluctuation process $\bQ=(\bQ(t))_{t\geq 0}$ independent of $\bL$, the \textit{perturbed insurance risk process (perturbed reserve process)} is given by 
\begin{equation}\label{eq:perturbed-reserve}
 \bR_u^{\bQ}(t)=u\ba+t\bp-\bL(t)+\bQ(t),
 \qquad t\geq 0,
\end{equation}
and the individual perturbed risk process of line $j$ is
\[
 R_{u,j}^{\bQ}(t)=ua_j+tp_j-L_j(t)+Q_j(t),\qquad t\geq 0.
\]
Of course, $\bQ(t)\equiv 0$ for $t\geq 0$ gives back the classical insurance risk process $\bR_u$, which is a special case of this more general model. 

The insolvency event is characterized by a Borel set $\Lambda \subset \mathbb{R}^d$, and the corresponding \textit{finite-time ruin probability} over the horizon $T>0$ is
\begin{eqnarray*}
 \psi_u^{\bQ}(\Lambda,T)
 &=&\P\!\left(
 \bR_u^{\bQ}(t)\in\Lambda
 \text{ for some }t\in[0,T]
 \right)\\
 &=&\P\!\left(
   u\ba+t\bp-\bL(t)+\bQ(t)\in\Lambda
   \text{ for some }t\in[0,T]
   \right)\\
 &=&\P\!\left(
   \bL(t)-\bQ(t)-t\bp\in u\ba-\Lambda
   \text{ for some }t\in[0,T]
   \right).
\end{eqnarray*}
Hence, equivalently, ruin occurs when the claim surplus enters the set $u\ba-\Lambda$. 
We typically consider ruin sets of the form
\begin{equation*}\label{eq:ruin-cone}
  \Lambda=-C_{\bx},\qquad \text{where} \qquad  C_{\bx}=\bx+C,
\end{equation*}
and $\bx\in\R_+^d$ is a fixed displacement and $C$ is an \textit{increasing cone}. Note that a set $C\subset\R^d$ is called \textit{increasing} if $\bc\in C$ and $\bz\in\R_+^d$ imply $\bz+\bc\in C$; it is a \textit{cone} if $\lambda C=C$ for every $\lambda >0$.  The shift $\bx$ can represent, for example, a fixed deductible, a regulatory buffer, or a deterministic displacement of the insolvency boundary. The class of increasing cones includes orthants, unions of orthants, increasing half-spaces, and intersections or unions of homogeneous group constraints; concrete examples are given in \Cref{ex:ruin-geometries}. Note that these insolvency sets are more general than most of those previously investigated in the literature and include, in particular, the sets considered in \cite{samorodnitsky:sun:2016}, where it was additionally assumed that $\bzero \in \Lambda$, that $\Lambda$ is open, and that $\Lambda^c$ is convex.

The event in claim space that is relevant after scaling by the total capital is obtained by translating $C$ by the allocation vector $\ba$ and defines the \textit{normalized claim region} corresponding to the insolvency region $\Lambda$  as 
\begin{equation*}\label{eq:Aib}
   B_{\Lambda}:=C_{\ba}\cap\R_+^d.
\end{equation*}
For fixed $T$ and premium $\bp$, the normalized premium term
$t\bp/u$ converges to zero uniformly over $t\in[0,T]$ as $u\to\infty$.
Consequently, the premium vector does
not enter the normalized claim region or its leading asymptotic limit. 
For the same reason, the normalized claim region is independent of the shift $\bx$.

The position of this set in the subcone hierarchy $\bbO_d^{(i)}$, $i=1,\ldots,d$, determines the relevant subcone and the scale of the ruin probability.
The asymptotic ruin probability is then determined by three quantities: The index $i$ identifies the first subcone whose limit measure $\mu_i$ has mass on the normalized claim region $B_{\Lambda}$. The inverse scaling function $b_i^{\leftarrow}(u)$ determines the speed of decrease of the ruin probability. Finally, $f_i(T)\mu_i(B_{\Lambda})$ gives the limit value: $f_i(T)$ describes the relevant large-jump mechanism, while $\mu_i(B_{\Lambda})$ captures the extremal dependence, capital allocation, and the geometry of the ruin region. The result is summarised as follows.

\begin{theorem}\label{thm:ruin} 
Fix $i\in\mathbb I_d$ and $T>0$. Let the perturbed insurance risk process $\bR_u^{\bQ}$ be defined as in \eqref{eq:perturbed-reserve}. 
Suppose Assumption~\ref{ass:riskmodel} holds with a positive measurable function $f_i:(0,\infty)\to(0,\infty)$ satisfying $f_i(1)=1$, such  that for $T>0$,
\[
 \bL(T)\in\MRV\big(\alpha_i,b_i,f_i(T)\mu_i,\bbO_d^{(i)}\big)
\]
and $\overline{\bQ}_T:=\sup_{0\leq t\leq T}\|\bQ(t)\|_\infty$  satisfies 
$\E\vert \overline{\bQ}_T \vert^{\alpha_i+\gamma}<\infty$ for some $\gamma>0$.
Let  $$\Lambda=-(\bx+C)$$ with $\bx\in\R_+^d$ and $C$ an increasing cone, be  the insolvency set, and let
$B_{\Lambda}:=C_{\ba}\cap\R_+^d \in \mathcal{B}(\bbO_d^{(i)})$ be the normalized claim region. Assume that $B_{\Lambda}$ is bounded away from the cone $\CA_d^{(i-1)}$ and is a $\mu_i$-continuity set. Then 
\begin{equation*} % \label{eq:main-limit}
 \lim_{u\to\infty}b_i^{\leftarrow}(u)\psi_u^{\bQ}(\Lambda,T)
 =f_i(T)\mu_i(B_{\Lambda}),
\end{equation*}
which is independent of $\bx$.
If $0<\mu_i(B_{\Lambda})<\infty$, then
\begin{equation*} % \label{eq:endpoint-equivalence}
 \psi_u^{\bQ}(\Lambda,T)
 \sim \P\!\left(\bL(T)\in uB_{\Lambda}\right),
 \qquad u\to\infty.
\end{equation*}
\end{theorem}

\begin{proof}
The proof uses the sandwich principle.
Since $C$ is a cone,
\begin{equation*} 
 u(\ba+C)=u\ba+C
 \qquad \text{ and } \quad
 uB_{\Lambda}=(u\ba+C)\cap\R_+^d.
\end{equation*}
For the lower bound, first note that
\begin{align*}
 &\{\bL(T)-\bQ(T)-T\bp-\bx\in uB_{\Lambda}\}\\
 &\quad\subset
 \{\bL(T)-\bQ(T)-T\bp-\bx\in u(\ba+C)\}\\
 &\quad=
 \{u\ba+T\bp-\bL(T)+\bQ(T)\in-(\bx+C)\}\\
 &\quad 
 \subset \{   \bR_u^{\bQ}(t)\in\Lambda
   \text{ for some }t\in[0,T]\}.
\end{align*}
Hence,
\begin{equation*} 
 \psi_u^{\bQ}(\Lambda,T)
 \geq
 \P\!\left(\bL(T)-\bQ(T)-T\bp-\bx\in uB_{\Lambda}\right).
\end{equation*}
 A consequence of $\bL(T)\in\MRV\big(\alpha_i,b_i,f_i(T)\mu_i,\bbO_d^{(i)}\big)$
and \cite[Lemma 2]{das:fasen:2027} is that $\bL(T)-\bQ(T)-T\bp-\bx\in\MRV\big(\alpha_i,b_i,f_i(T)\mu_i,\bbO_d^{(i)}\big)$ as well. Since 
$B_{\Lambda}\in \mathcal{B}(\bbO_d^{(i)})$ is bounded away from  $\CA_d^{(i-1)}$ and is a $\mu_i$-continuity set, we obtain the following
\begin{eqnarray}\label{eq:lower-main}
 \liminf_{u\to\infty}b_i^{\leftarrow}(u)\psi_u^{\bQ}(\Lambda,T)
    &\geq& \liminf_{u\to\infty}b_i^{\leftarrow}(u)
 \P\!\left(\bL(T)-\bQ(T)-T\bp-\bx\in uB_{\Lambda}\right)\nonumber\\
    &=& f_i(T)\mu_i(B_{\Lambda}).
\end{eqnarray}
For the upper bound, suppose ruin occurs at some $t\in[0,T] $. Then
\[
 \bL(t)-\bQ(t)\in u\ba+(\bx+t\bp)+C\subset u\ba+C.
\]
The last inclusion holds because $\bx+t\bp\in\R_+^d$ and $C$ is increasing. Since $\bL$ is a subordinator,
$\bL(T)-\bL(t)\in\R_+^d$ and $u\ba+C$ is increasing, $\bL(t)-\bQ(t)\in u\ba+C$ implies $\bL(T)+\overline{\bQ}_T\bone \in u\ba+C$. Finally $\bL(T)+\overline{\bQ}_T\bone\in\R_+^d$ gives
\begin{equation*} 
 \{\bR_u^{\bQ}(t)\in\Lambda
   \text{ for some }t\in[0,T]\}
 \subset \{\bL(T)+\overline{\bQ}_T\bone \in u\ba+C\} =\{\bL(T)+\overline{\bQ}_T\bone \in uB_{\Lambda}\}.
\end{equation*}
Thus,
\begin{equation}\label{eq:upper-main}
 \limsup_{u\to\infty}b_i^{\leftarrow}(u)\psi_u^{\bQ}(\Lambda,T)
    \leq \limsup_{u\to\infty}b_i^{\leftarrow}(u)
 \P\!\left(\bL(T)+\overline{\bQ}_T\bone \in uB_{\Lambda}\right)=f_i(T)\mu_i(B_{\Lambda}),
\end{equation}
by $\bL(T)\in\MRV\big(\alpha_i,b_i,f_i(T)\mu_i,\bbO_d^{(i)}\big)$ and \citet[Lemma 2]{das:fasen:2027}. Combining \eqref{eq:lower-main} and \eqref{eq:upper-main} proves 
\begin{equation} \label{eq:main-limit}
 \lim_{u\to\infty}b_i^{\leftarrow}(u)\psi_u^{\bQ}(\Lambda,T)
 =f_i(T)\mu_i(B_{\Lambda}),
\end{equation} 
Moreover,
\begin{align*} 
 \lim_{u\to\infty}b_i^{\leftarrow}(u)\P\!\left(\bL(T)\in uB_{\Lambda}\right)=f_i(T)\mu_i(B_{\Lambda}).
\end{align*}
If this limit is positive, we directly obtain from \eqref{eq:main-limit} the final statement $\psi_u^{\bQ}(\Lambda,T)
 \sim \P\!\left(\bL(T)\in uB_{\Lambda}\right)$,
 $u\to\infty$.
\end{proof}

For the actuarial interpretation, it is important to note that the asymptotic behavior of the ruin probability is determined by the risk process at the terminal time $T$, whereas the behavior of the risk process on the interval $[0,T)$ has no influence. To be more precise, the leading approximation is obtained from the probability that the terminal cumulative-claim vector $\bL(T)$ enters the normalized claim region $uB_{\Lambda}$.
Additionally, the displacement vector $\bx$ does not appear in the limit. The limiting value depends only on the terminal time $T$, and the geometry of the normalized claim region $B_{\Lambda}$, which is determined by the cone $C$ and the capital allocation vector $\ba$.

\begin{remark}\label{rem:ruin-depth-selection}
When Assumption~\ref{ass:riskmodel} holds on all subcones, Theorem~\ref{thm:ruin} also provides a practical way to select the ruin probability rate. Starting from $i=1$, we choose the first subcone for which $B_{\Lambda}=(\ba+C)\cap\R_+^d$ is bounded away from the cone $\CA_d^{(i-1)}$, is a $\mu_i$-continuity set, and satisfies $\mu_i(B_{\Lambda})>0$. This smallest index
$i_*(B_{\Lambda})$ is called  the \textit{ruin depth}  of $B_{\Lambda}$ and at this index $i_*(B_{\Lambda})$ we have
\[
 \lim_{u\to\infty}b_{i_*(B_{\Lambda})}^{\leftarrow}(u)\psi_u^{\bQ}(\Lambda,T)
 = f_{i_*(B_{\Lambda})}(T)\mu_{i_*(B_{\Lambda})}(B_{\Lambda})>0.
\]
Each preceding subcone $\bbO_d^{(k)}$, $k=1,\ldots,i_*(B_{\Lambda})-1 $  either fails to separate $B_{\Lambda}$ from its cone $\CA_d^{(k-1)}$ or assigns zero limiting mass to $B_{\Lambda}$. Hence, not only the claim surplus process but also the geometry of the ruin region determines the ruin probability rate.
\end{remark}

A very special and important case is the Brownian perturbation given in the next corollary.

\begin{corollary}[Brownian perturbations]\label{cor:brownian}
Let $\bB$ be a $k$-dimensional standard Brownian motion, let $\bD$ be a fixed $\R^{d\times k}$ matrix, and let $\bq:[0,T]\to\R^d$ be measurable and bounded. Then \Cref{thm:ruin}  holds for
\[
  \bQ(t)=\bD\bB(t)+\bq(t), \qquad t\geq0 .
\]
In particular, the finite-time heavy-tailed ruin asymptotic is unchanged by a multivariate Brownian perturbation.
\end{corollary}
The proof follows directly  from \Cref{thm:ruin}, the Gaussian marginal distributions of a Brownian motion, and the martingale property of a Brownian motion in combination with  Doob's inequality.

\begin{remark}[Solvency-capital approximation] \label{remark:solvency}
The asymptotic behavior of the ruin probability allows the calculation of the solvency capital for a given ruin probability \linebreak $p_{\mathrm{ruin}}\in(0,1)$.
For a specified ruin region $\Lambda$, horizon $T$, and allocation vector $\ba$, define the \textit{solvency capital}, the total capital required to guarantee solvency with probability at least $1-p_{\mathrm{ruin}}$, by
\begin{equation*}
 u_{p_{\mathrm{ruin}}}(\Lambda,T,\ba)
 :=\inf\{u>0:\psi_u^{\bQ}(\Lambda,T)\leq p_{\mathrm{ruin}}\}.
\end{equation*}
This standard ruin-based construction defines the required solvency capital as the generalized inverse of the ruin probability at a prescribed tolerance level; see \cite{hult2006heavy}. 
Since $B_{\Lambda}$ depends only on $\ba$ and $C$, the definition
\[
 K_i(C,T,\ba):=f_i(T)\mu_i(B_{\Lambda})\geq 0
\] 
is well-defined. If $\mu_i(B_{\Lambda})>0$, then
Theorem~\ref{thm:ruin} gives
$\psi_u^{\bQ}(\Lambda,T)\sim K_i(C,T,\ba)/b_i^{\leftarrow}(u)$ as $u\to\infty$.
If the function $b_i$ is chosen such that it is eventually continuous and strictly increasing, then the 
monotonicity in $u$ of the right-hand side and the asymptotic-inverse theorem for regularly varying functions \citep[Theorem~1.5.12]{bingham:goldie:teugels:1989} yield 
\begin{equation}\label{eq:capital-quantile}
 u_{p_{\mathrm{ruin}}}(\Lambda,T,\ba)
 \sim b_i\left(\frac{K_i(C,T,\ba)}{p_{\mathrm{ruin}}}\right), \qquad \text{ as } \;\; p_{\mathrm{ruin}}\downarrow0.
\end{equation}
This approximates the solvency capital at a high probability level $1-p_{\mathrm{ruin}}$. 

The insurance company is also interested in minimizing its solvency capital.  This can ideally be achieved by selecting the allocation principle  $\ba$ which minimizes $u_{p_{\mathrm{ruin}}}(\Lambda,T,\ba)$.
 Since this quantity is not available explicitly and only its asymptotic behavior is known, the \textit{optimal allocation} principle $\ba^*$ is chosen as
\begin{eqnarray}\label{eq:solvency-alloc}
    \ba^*=\argmin_{\ba\in \Delta_d^\circ} \left\{b_i\left(K_i(C,T,\ba)/p_{\mathrm{ruin}}\right)\right\} \;\text{ with }\;
    \Delta_d^\circ:=\{\ba\in(0,\infty)^d:\sum_{j=1}^da_j=1\}.
\end{eqnarray}
We return to this topic in \Cref{sec:simulation}, where concrete examples and simulation studies are used to illustrate the theoretical findings.
\end{remark}

\section{Key ruin regions} \label{ex:ruin-geometries} 
In the following examples, we list several ruin regions covered by Theorem~\ref{thm:ruin}. In every case, the set $B_{\Lambda}$ is assumed to be a $\mu_i$-continuity set for the relevant limit measure $\mu_i$, $a_j>0$ for all $j\in\mathbb{I}_d$, and  $\bx$ is equal to $\bzero$, because Theorem \ref{thm:ruin} shows that this has no effect on the limit measure.

\begin{enumerate}[(a)]
\item \textit{All lines are ruined.} Let
\[
 \Lambda_{\mathrm{all}}=\{\bz\in\R^d:z_j<0\text{ for all }j\}.
\]
Then $C=(0,\infty)^d$, $\Lambda_{\mathrm{all}}=-C$, and
\[
 B_{\Lambda_{\mathrm{all}}}=C_{\ba}\cap\R_+^d
 =\{\bz\in\R_+^d:z_j>a_j\text{ for all }j\}
 \in\mathcal{B}(\bbO_d^{(d)}).
\]
Thus
\[
 \lim_{u\to\infty}b_d^{\leftarrow}(u)\psi_u^{\bQ}(\Lambda_{\mathrm{all}},T)
 = f_d(T)\mu_d(B_{\Lambda_{\mathrm{all}}}).
\]

\item \textit{At least $i$ business lines are ruined.} Define
\[
 \Lambda_i=\{\bz\in\R^d:\#\{j\in\mathbb{I}_d:z_j<0\}\geq i\}.
\]
Here $C=\{\bz\in \R^d:z_{(i)}>0\}$ and
\begin{equation*} 
 B_{\Lambda_i}=\bigcup_{S\subset\mathbb I_d,\,|S|=i}
 \{\bz\in\R_+^d:z_j>a_j,\ j\in S\}
 \in\mathcal{B}(\bbO_d^{(i)}).
\end{equation*}
Consequently,
\[
 \lim_{u\to\infty}b_i^{\leftarrow}(u)\psi_u^{\bQ}(\Lambda_i,T)
 = f_i(T)\mu_i(B_{\Lambda_i}).
\]
Note that  $\Lambda_d=\Lambda_{\mathrm{all}}$ from part (a).

If $\alpha_1<\cdots<\alpha_d$, the probabilities of $\Lambda_1,\ldots,\Lambda_d$  decay at genuinely different powers. If $i>1$ we have zero mass $\mu_k(B_{\Lambda_i})$ for $k=1,\ldots,i-1$ and regular variation on $\bbO_d^{(i)}$ identifies the correct non-zero scale. 
In particular classical regular variation on $\bbO_d^{(1)}$ gives $\mu_1(B_{\Lambda_i})=0$. 
\item \textit{Capital transfers.} Suppose that a fraction $\omega_j\in[0,1]$ of a positive reserve in business line $j$ can be transferred to cover losses in other business lines. The ruin set under such transfers is
\[
 \Lambda_{\mathrm{tr}}
 =\left\{\bz\in \R^d:\sum_{j=1}^d\omega_j(z_j)_+
 <\sum_{j=1}^d(z_j)_-\right\}.
\]
Then $\Lambda_{\mathrm{tr}}=-C$ where
\[
 C=\left\{\bz\in \R^d:\sum_{j=1}^d\omega_j(z_j)_-
 <\sum_{j=1}^d(z_j)_+\right\},
\]
is an increasing cone. Typically, $B_{\Lambda_{\mathrm{tr}}}=C_{\ba}\cap\R_+^d$ is already in $\bbO_d^{(1)}$, giving
\[
 \lim_{u\to\infty}b_1^{\leftarrow}(u)\psi_u^{\bQ}(\Lambda_{\mathrm{tr}},T)= f_1(T)\mu_1(B_{\Lambda_{\mathrm{tr}}}).
\]
Related multivariate transfer regions are considered in \cite{hult2006heavy,li2015asymptotic}.

\item \textit{Ruin events for reinsurance portfolios.} Consider $d$ separate insurance companies indexed by $\mathbb I_d$ and $r$ reinsurance companies. Reinsurer $\ell$ is interested in the solvency of the insurance companies in a portfolio $S_\ell\subset\mathbb I_d$ with $S_\ell\neq\emptyset$. We assume that $S_1,\ldots,S_r$  form a partition of $\mathbb I_d$, implying that every insurance company is assigned to exactly one reinsurance company and can contribute to the loss incurred by the corresponding reinsurer. Three natural ruin events in this context are important for the reinsurers.
\begin{enumerate}[(i)]
\item \textit{Aggregate deficit for every reinsurer.} The portfolio of each reinsurer has a deficit:
\[
 \Lambda_{\mathrm{grp}}
 =\left\{\bz\in\R^d:\sum_{j\in S_\ell}z_j<0,
 \ \ell=1,\ldots,r\right\}.
\]
The corresponding normalized claim region is
\[
 B_{\Lambda_{\mathrm{grp}}}
 =\left\{\bz\in\R_+^d:
 \sum_{j\in S_\ell}z_j>\sum_{j\in S_\ell}a_j,
 \ \ell=1,\ldots,r\right\}
 \in\mathcal{B}(\bbO_d^{(r)}),
\]
and hence,
\[
 \lim_{u\to\infty} b_r^{\leftarrow}(u)\psi_u^{\bQ}(\Lambda_{\mathrm{grp}},T)
 = f_r(T)\mu_r(B_{\Lambda_{\mathrm{grp}}}).
\]

\item \textit{At least one insolvent insurer for every reinsurer.} Each reinsurance portfolio contains at least one insolvent insurance company:
\[
 \Lambda_{\mathrm{one/each}}
 =\left\{\bz\in\R^d:\min_{j\in S_\ell}z_j<0,
 \ \ell=1,\ldots,r\right\}.
\]
Then
\[
 B_{\Lambda_{\mathrm{one/each}}}
 =\left\{\bz\in\R_+^d:
 \max_{j\in S_\ell}(z_j-a_j)>0,
 \ \ell=1,\ldots,r\right\}
 \in\mathcal{B}(\bbO_d^{(r)}),
\]
and therefore
\[
 \lim_{u\to\infty}b_r^{\leftarrow}(u)\psi_u^{\bQ}(\Lambda_{\mathrm{one/each}},T)
 = f_r(T)\mu_r(B_{\Lambda_{\mathrm{one/each}}}).
\]

\item \textit{Complete failure of one reinsurance portfolio.} All insurance companies in at least one reinsurance portfolio are insolvent. Define
\[
 \Lambda_{\mathrm{one\ group}}
 =\bigcup_{\ell=1}^r
 \{\bz\in\R^d:\max_{j\in S_\ell}z_j<0\} 
 \qquad \text{ and } \qquad s_*:=\min_{1\leq\ell\leq r}|S_\ell|.
\]
Then
\[
 B_{\Lambda_{\mathrm{one\ group}}}
 =\bigcup_{\ell=1}^r
 \{\bz\in\R_+^d:\min_{j\in S_\ell}(z_j-a_j)>0\}
 \in\mathcal{B}(\bbO_d^{(s_*)}).
\]
Because all $a_j>0$, this set is bounded away from  $\CA_d^{(s_*-1)}$, and hence
\[
 \lim_{u\to\infty}b_{s_*}^{\leftarrow}(u)\psi_u^{\bQ}(\Lambda_{\mathrm{one\ group}},T)
 = f_{s_*}(T)\mu_{s_*}(B_{\Lambda_{\mathrm{one\ group}}}).
\]
\end{enumerate}

\item \textit{Guarantee-fund shortfall.} Following \citet[Section~4.2]{hult2006heavy}, suppose that a fraction $\gamma\in[0,1]$ of the total initial capital $u$ is held in a common guarantee fund. The remaining capital $(1-\gamma)u$ is allocated among the companies according to $\ba$ and cannot be transferred directly between them. The fund is used to cover negative company reserves, and ruin occurs when their aggregate deficit exceeds the fund capital $\gamma u$. Write the finite-horizon ruin probability as
\[
 \psi_u^{\mathrm{gf}}(\gamma,T)
 :=\P\!\left(
 \sum_{j=1}^d
 \bigl(L_j(t)-Q_j(t)-p_jt-(1-\gamma)ua_j\bigr)_+>\gamma u
 \text{ for some }t\in[0,T]
 \right),
\]
and its corresponding normalized claim region is
\begin{equation}\label{eq:guarantee-fund-region}
 B_{\mathrm{gf}}(\gamma)
 =\left\{\bz\in\R_+^d:
 \sum_{j=1}^d\bigl(z_j-(1-\gamma)a_j\bigr)_+>\gamma\right\}
 \in\mathcal{B}(\bbO_d^{(1)}).
\end{equation}
Thus, whenever $B_{\mathrm{gf}}(\gamma)$ is a $\mu_1$-continuity set with positive mass,
\[
 \lim_{u\to\infty}b_1^{\leftarrow}(u)\psi_u^{\mathrm{gf}}(\gamma,T)=
  f_1(T)\mu_1\!\left(B_{\mathrm{gf}}(\gamma)\right).
\]
For $\gamma=0$, this reduces to the ruin of at least one company without capital transfers; for $\gamma=1$, all initial capital is held in the fund and ruin occurs when the aggregate deficit exceeds that capital.
For small $\gamma$ the leading depth-one scale is still correct, because one sufficiently large line can exhaust the fund, but moderate capital levels may also show visible contributions from multi-line deficits. 
\end{enumerate}

\section{Explicit ruin probabilities in examples}\label{sec:examples}

We now compute the asymptotic ruin probabilities for a few ruin sets under different model assumptions for claim arrival and extremal dependence among claims. 
We concentrate on ruin sets discussed in \Cref{ex:ruin-geometries}(b), first considering the event that at least $i$ business lines are ruined,
\[
 \Lambda_i=\{\bz\in\R^d:\#\{j:z_j<0\}\geq i\},
\]
and its normalized claim region 
\[
 B_{\Lambda_i}=\bigcup_{{S\subset\mathbb I_d,|S|=i}}
 \{\bz\in\R_+^d:z_j>a_j\text{ for every }j\in S\}.
\]
Here each subset $S$ represents a particular collection of $i$ failed lines. We give particular attention to $i\ge 2$, since previous results in the literature are limited in these cases.
From \Cref{ex:ruin-geometries}(d), we also use grouped ruin regions for $r$ disjoint reinsurance portfolios $S_1,\ldots,S_r$, and write
\[
 \overline{a}_\ell=\sum_{j\in S_\ell}a_j\qquad \text{ and }\qquad
 s_*=\min_{1\leq\ell\leq r}|S_\ell|.
\]
We focus on three grouped ruin sets: $\Lambda_{\mathrm{grp}}$, $\Lambda_{\mathrm{one/each}}$, and $\Lambda_{\mathrm{one\ group}}$. %For the direct business-line examples, each normalized claim region is a Borel set in a subcone $\E_d^{(i)}$, and the relevant value of $i$ determines the scaling function. 
 For the claims process $\bL$, we treat four benchmark models: independent subordinators across the $d$ lines, and common compound Poisson arrivals with $d$-dimensional claim vectors whose coordinates are (i) independent, (ii) Gaussian-copula dependent, or (iii) Marshall--Olkin dependent. Finally, in Section~\ref{sec:bipartite-insurance-network}, we investigate a bipartite insurance network model. %place the $d$ business-line claim coordinates on the object side of a bipartite insurance network and map them to $r$ ring-fenced insurance portfolios through a nonnegative weighted adjacency matrix.

The calculations below show how the same ruin geometry can have different probability scales under different extremal dependence and arrival mechanisms. A compendium of the computed limits is available in \Cref{app:summary-rates}.
 For the business-line examples and for sets with all $a_j>0$, the relevant normalized claim regions are $\mu_i$-continuity sets for the corresponding limit measure $\mu_i$; the continuity condition for the network pre-image is stated separately in Section~\ref{sec:bipartite-insurance-network}. Moreover, for explicit computation, depending on the model, we assume that either the marginal distributions of the claims or the marginal L\'evy measures have an asymptotic or exact power-law (Pareto-type) tail behavior and are tail equivalent to one another with common tail index $\alpha>0$. 
%; the explicit assumptions are provided below. 

\subsection{Independence across claim lines}\label{sec:independent-models}
We consider two claim processes with independence across different claim lines: the first uses independent subordinators for each line, and the second uses a common Poisson arrival process with independent claims among business lines (coordinates) at each arrival.
\subsubsection{Independent component subordinators}\label{sec:indsub}

Let $L_1,\ldots,L_d$ be independent subordinators representing $d$ lines of claims. For line $j$, let $\Pi_j$ denote the marginal L\'evy measure and suppose its upper tail satisfies
\begin{equation*}
 \Pi_j((u,\infty))\sim c_j u^{-\alpha},
   \qquad u\to\infty \qquad (c_j>0),
\end{equation*}
for a common $\alpha>0$. 
The one-dimensional one-large-jump property for a subordinator with regularly varying L\'evy tail measure \citep[see, e.g.,][Section~3.1]{hult2006heavy} gives, for each $j$,
\[
 \P\!\left(L_j(T)>ux\right)\sim T\Pi_j((ux,\infty))
 \sim Tc_jx^{-\alpha}u^{-\alpha},  \qquad u\to\infty.
\]
Independence yields, for $i\in\mathbb I_d$ and every $S\subseteq \mathbb I_d$ with $|S|=i$,
\[
 \P\!\left(L_j(T)>ua_j,\ j\in S\right)
 \sim T^iu^{-i\alpha}\prod_{j\in S}c_ja_j^{-\alpha},  \qquad u\to\infty,
\]
giving the time factor $f_i(T)=T^i$ and the limit measure $\mu_i^L$ as in Proposition~\ref{prop:cpp-separate-jumps}. Thus, by Theorem~\ref{thm:ruin},
\begin{equation}\label{eq:independent-lines-limit}
 \lim_{u\to\infty}u^{i\alpha}\psi_u^{\bQ}(\Lambda_i,T)
 =
 T^i\sum_{S\subset\mathbb I_d,\,|S|=i}
 \prod_{j\in S}c_ja_j^{-\alpha},
\end{equation}
and in particular,
\[
 \lim_{u\to\infty}u^{d\alpha}\psi_u^{\bQ}(\Lambda_{\mathrm{all}},T)
 = T^d\prod_{j=1}^dc_ja_j^{-\alpha}.
\]
For the reinsurance portfolios in \Cref{ex:ruin-geometries}(d), a routine calculation of the limit measures for the ruin sets gives
\begin{align*}
 &\lim_{u\to\infty}u^{r\alpha}\psi_u^{\bQ}(\Lambda_{\mathrm{grp}},T)
 =T^r\prod_{\ell=1}^r
 \Big(\overline{a}_\ell^{-\alpha}\sum_{j\in S_\ell}c_j\Big),\\
 &\lim_{u\to\infty}u^{r\alpha}\psi_u^{\bQ}(\Lambda_{\mathrm{one/each}},T)
 =T^r\prod_{\ell=1}^r
 \Big(\sum_{j\in S_\ell}c_ja_j^{-\alpha}\Big),\\
 &\lim_{u\to\infty}u^{s_*\alpha}\psi_u^{\bQ}(\Lambda_{\mathrm{one\ group}},T)
 =T^{s_*}
 \sum_{\ell:\,|S_\ell|=s_*}\prod_{j\in S_\ell}c_ja_j^{-\alpha}.
\end{align*}

\subsubsection{Common Poisson arrivals with independent claims across lines}\label{sec:indcpp}
In this model we consider common Poisson arrivals of $d$-dimensional claims so that 
\[
 \bL(t)=\sum_{k=1}^{N_\lambda(t)}\bZ^{(k)},
\]
where $N_\lambda$ is a Poisson process with intensity $\lambda>0$. Each claim is a random vector
$\bZ^{(k)}=(Z_1^{(k)},\ldots,Z_d^{(k)})^\top\in\R_+^d$. The claims are independent and identically distributed, and the coordinates of each vector are mutually independent. Assume
$\P(Z_j^{(k)}>u)\sim {c}_j u^{-\alpha}$  as $u\to\infty$ for every $j\in \bI_d$. %Consequently, the marginal L\'evy measures satisfy $\Pi_j((u,\infty))\sim \lambda c_ju^{-\alpha}$.
For $i\in\mathbb I_d$ and every $S\subseteq\mathbb I_d$ with $|S|=i$, Proposition~\ref{prop:poisson-moment-transfer} gives
\begin{align*}
 \P\!\left(L_j(T)>ua_j,\ j\in S\right)
 & \sim \E[N_\lambda(T)^i]u^{-i\alpha}
 \prod_{j\in S}c_ja_j^{-\alpha},\qquad u\to\infty,\\
 & =: u^{-i\alpha} f_i(T) \mu_i(\{\bx\in \R_+^d: x_j> a_j, j\in S\})
\end{align*}
following the normalization convention of Assumption~\ref{ass:riskmodel}, with $f_i(T)={\E[N_\lambda(T)^i]}/{\E[N_\lambda(1)^i]}$, 
and the factor $\E[N_\lambda(1)^i]$ being incorporated into the limit measure $\mu_i$. Theorem~\ref{thm:ruin} then yields
\begin{equation}\label{eq:cpp-independent-ruin}
 \lim_{u\to\infty}u^{i\alpha}\psi_u^{\bQ}(\Lambda_i,T)
 =
 \E[N_\lambda(T)^i]
 \sum_{S\subset\mathbb I_d,\,|S|=i}
 \prod_{j\in S}c_ja_j^{-\alpha}.
\end{equation}
For the reinsurance portfolios in \Cref{ex:ruin-geometries}(d), a routine computation again gives
\begin{align*}
 &\lim_{u\to\infty}u^{r\alpha}\psi_u^{\bQ}(\Lambda_{\mathrm{grp}},T)
 =\E[N_\lambda(T)^r]\prod_{\ell=1}^r
 \Big(\overline{a}_\ell^{-\alpha}\sum_{j\in S_\ell}c_j\Big),\\
 &\lim_{u\to\infty}u^{r\alpha}\psi_u^{\bQ}(\Lambda_{\mathrm{one/each}},T)
 =\E[N_\lambda(T)^r]\prod_{\ell=1}^r
 \Big(\sum_{j\in S_\ell}c_ja_j^{-\alpha}\Big),\\
 &\lim_{u\to\infty}u^{s_*\alpha}\psi_u^{\bQ}(\Lambda_{\mathrm{one\ group}},T)
 =\E[N_\lambda(T)^{s_*}]
 \sum_{\ell:\,|S_\ell|=s_*}\prod_{j\in S_\ell}c_ja_j^{-\alpha}.
\end{align*}
Recall that the independent-subordinator time factor in Assumption~\ref{ass:riskmodel} is $T^i$; in the common-arrival model, it becomes
 ${\E[N_\lambda(T)^i]}/{\E[N_\lambda(1)^i]}$. These factors agree for $i=1$; for $i\geq2$, the Poisson moment accounts for the common number of claim arrivals. For standard Pareto margins, as in the Gaussian copula case with $\rho=0$ considered in Section~\ref{sec:gaussian}, we have $c_j=1$ for every $j$, so \eqref{eq:cpp-independent-ruin} directly gives the corresponding formula.

\subsection{Gaussian-copula claims}\label{sec:gaussian}

Let $\Phi$ be the standard normal distribution function, let $I_d$ be the $d\times d$ identity matrix, let $\bone$ be the $d$-vector of ones, and let $\alpha>0$. For $-1/(d-1)<\rho<1$, define the equicorrelation matrix
\[
 \Sigma_\rho=(1-\rho)I_d+\rho\bone\bone^\top,
\]
and let $\Phi_{\Sigma_\rho}$ denote the distribution function of a centered $d$-dimensional Gaussian random vector with covariance (and hence in this case correlation) matrix $\Sigma_\rho$. We consider a claim vector $\bZ=(Z_1,\ldots,Z_d)^\top$ with joint distribution
\[
 \P(\bZ\leq\bz)
 =\Phi_{\Sigma_\rho}\!\left(
   \Phi^{\leftarrow}(1-z_1^{-\alpha}),\ldots,
   \Phi^{\leftarrow}(1-z_d^{-\alpha})
 \right),
 \qquad \bz=(z_1,\ldots,z_d)^\top\in[1,\infty)^d.
\]
Thus $\bZ$ has Pareto margins with shape parameter $\alpha$, $\P(Z_j>z)=z^{-\alpha}$ for $z\geq1$, and a Gaussian copula with common pairwise correlation $\rho$. Following \citet{das:fasen:2024}, define, for $i\ge 2$,
\begin{equation}\label{eq:gaussian-parameters}
 \gamma_{i,\rho}=\frac{i}{1+(i-1)\rho},
 \qquad h_{i,\rho}=\frac1{1+(i-1)\rho},
\end{equation}
and the normalising function
\begin{equation*}\label{eq:gaussian-rate}
 b_{i,\rho}^{\leftarrow}(u)
 =(2\pi)^{-\gamma_{i,\rho}/2}
 (2\alpha\log u)^{(i-\gamma_{i,\rho})/2}
 u^{\alpha\gamma_{i,\rho}}.
\end{equation*}
Then, by \citet[Theorem 1 and Example 2]{das:fasen:2024}, we have \linebreak $\bZ\in \MRV(\alpha\gamma_{i,\rho}, b_{i,\rho}, \mu_{i,\rho}, \E_d^{(i)})$ for $i=2,\ldots, d$. 
On a rectangular set indexed by $S\subset\mathbb I_d$ with $|S|=i$, the measure $\mu_{i,\rho}$ is given by
\begin{equation*}
 \mu_{i,\rho}(\{\bz\in\R_+^d:z_j>x_j,\ j\in S\})
 =\Upsilon_{i,\rho}
 \prod_{j\in S}x_j^{-\alpha h_{i,\rho}},
\end{equation*}
where
\begin{equation*}\label{eq:upsilon-equicorr}
 \Upsilon_{i,\rho}
 =\frac{(1+(i-1)\rho)^{i-1/2}}
 {(2\pi)^{i/2}(1-\rho)^{(i-1)/2}}.
\end{equation*}

Now, let $N_\lambda$ be a Poisson process with intensity $\lambda>0$, independent of the i.i.d. copies $\bZ^{(k)}$ of $\bZ$, and set
$\bL(t)=\sum_{k=1}^{N_\lambda(t)}\bZ^{(k)}$.
\begin{itemize}
    \item[(i)] For $\rho>0$, the inequalities $\gamma_{i,\rho}<\gamma_{\ell,\rho}+\gamma_{i-\ell,\rho}$ for $1\leq \ell<i$ hold, which is precisely the strict subadditivity condition \eqref{eq:strict-subadd} after multiplication by $\alpha$. Proposition~\ref{prop:levy-transfer} therefore yields a linear factor $\lambda T$. Then, as a consequence of Theorem~\ref{thm:ruin}, for $i=2,\ldots,d$, we have
\begin{equation}\label{eq:gaussian-ruin-limit}
 \lim_{u\to\infty} b_{i,\rho}^{\leftarrow}(u)\psi_u^{\bQ}(\Lambda_i,T)
    =
 \lambda T\,\Upsilon_{i,\rho}
 \sum_{S\subset\mathbb I_d,\,|S|=i}
 \prod_{j\in S}a_j^{-\alpha h_{i,\rho}}.
\end{equation}
For $i=d$, the sum in \eqref{eq:gaussian-ruin-limit} contains only $S=\mathbb I_d$ and gives the all-lines ruin constant.

The same computation gives the reinsurance portfolio constants: For $r\geq2$,
\begin{align*}
 &\lim_{u\to\infty}b_{r,\rho}^{\leftarrow}(u)\psi_u^{\bQ}(\Lambda_{\mathrm{grp}},T)
 =\lambda T\,\Upsilon_{r,\rho}
 \prod_{\ell=1}^r |S_\ell|\overline{a}_\ell^{-\alpha h_{r,\rho}},\\
 &\lim_{u\to\infty}b_{r,\rho}^{\leftarrow}(u)\psi_u^{\bQ}(\Lambda_{\mathrm{one/each}},T)
 =\lambda T\,\Upsilon_{r,\rho}
 \prod_{\ell=1}^r\sum_{j\in S_\ell}a_j^{-\alpha h_{r,\rho}}.
\end{align*}
Moreover, if $s_*\geq2$,
\[
 \lim_{u\to\infty}b_{s_*,\rho}^{\leftarrow}(u)\psi_u^{\bQ}(\Lambda_{\mathrm{one\ group}},T)
 =\lambda T\,\Upsilon_{s_*,\rho}
 \sum_{\ell:\,|S_\ell|=s_*}\prod_{j\in S_\ell}a_j^{-\alpha h_{s_*,\rho}},
\]
where only portfolios of smallest size contribute at this scale. If $s_*=1$, the event has non-trivial measure on the biggest subcone $\E_d^{(1)}$ and the  limit is
\[
 \lim_{u\to\infty}u^\alpha\psi_u^{\bQ}(\Lambda_{\mathrm{one\ group}},T)
 =\lambda T
 \sum_{\ell:\,S_\ell=\{j_\ell\}}a_{j_\ell}^{-\alpha}.
\]
The aggregate-deficit formula uses the threshold $\overline{a}_\ell$ for the coordinate selected from portfolio $S_\ell$, whereas the one-insolvent-insurer formula uses the individual thresholds $a_j$.

\item[(ii)] The strict subadditivity condition \eqref{eq:strict-subadd} used above to obtain \eqref{eq:gaussian-ruin-limit} fails at $\rho=0$, where $\gamma_{i,0}=i$ and the coordinates of each claim are independent. In this case, from \eqref{eq:cpp-independent-ruin} we have,
\[
 \lim_{u\to\infty}u^{i\alpha}\psi_u^{\bQ}(\Lambda_i,T)
 = \E[N_\lambda(T)^i]
 \sum_{|S|=i}\prod_{j\in S}a_j^{-\alpha}.
\]
\item[(iii)] For $-1/(d-1)<\rho<0$, on the other hand \eqref{eq:gaussian-parameters} gives $\gamma_{i,\rho}>i$. Hence a simultaneous $i$-coordinate exceedance within one claim is of smaller order than $u^{-i\alpha}$, and the L\'evy measure satisfies the (adapted null) condition \eqref{eq:adaptrv} on the higher subcones $\E_d^{(j)}, j>i$. Since its marginal L\'evy tails are asymptotic to $\lambda u^{-\alpha}$, Proposition~\ref{prop:cpp-separate-jumps} gives the product limit measure with time factor $T^i$. Applying Theorem~\ref{thm:ruin} and summing over the subsets $S$ with $|S|=i$ yields
\[
  \lim_{u\to\infty} u^{i\alpha}\psi_u^{\bQ}(\Lambda_i,T)
    = (\lambda T)^i
 \sum_{|S|=i}\prod_{j\in S}a_j^{-\alpha}.
\]
\end{itemize}
Thus positive, zero, and negative values, respectively of the Gaussian correlation $\rho$ produce  ruin events for which time factors are $\lambda T$, $\E[N_{\lambda}(T)^i]$, and $(\lambda T)^i$, respectively. The leading event is one common jump in the positive case and multiple jumps for $\rho\le 0$. This also illustrates why regular variation of the jump measure alone does not always determine regular variation of the aggregate on every subcone.

For a general positive definite correlation matrix $\Sigma$, there appears to be no comparably explicit formula, and such computations need to be done on a case-by-case basis.

\subsection{Marshall--Olkin copula claims}\label{sec:mo}

Let $\bZ=(Z_1,\ldots,Z_d)^\top$ have exact Pareto margins
$\P(Z_j>z)=z^{-\alpha}$ for $z\geq1$, and a Marshall--Olkin dependence structure. We use the two exchangeable specifications studied in \citet{das:fasen:2027,das:fasen:2026}; see also \citet{lin:li:2014}. For $i=1,\ldots,d$, set
\[
 \alpha_i^{=}=\big(2-2^{-(i-1)}\big)\alpha
 \qquad \text{ and }\qquad
 \alpha_i^{\propto}
 =\frac{\alpha}{d+1}
 \left(2d-\frac{d-i}{2^{i-1}}\right).
\]
Given $\bx\in (0,\infty)^d$ and any nonempty set $S\subset\mathbb I_d$, define $\bx_S:=(x_j)_{j\in S}$ with $x_j\geq1$. Let $x_{S,(1)}\geq\ldots\geq x_{S,(|S|)}$ denote the decreasing order statistics of the coordinates of $\bx_S$. For the equal-rate Marshall--Olkin dependence,
\begin{align*}
\P(Z_j>x_j,\ j\in S) 
 &=\prod_{j=1}^{|S|}x_{S,(j)}^{-\alpha2^{-(j-1)}} =:  \overline{\mu}_S^{=}(\bx_S),\\
 \intertext{and, for the proportional-rate Marshall--Olkin dependence,}
 \P(Z_j>x_j,\ j\in S) &=\prod_{j=1}^{|S|}
 x_{S,(j)}^{-\alpha(1-(j-1)/(d+1))2^{-(j-1)}} =:  \overline{\mu}_S^{\propto}(\bx_S).
\end{align*}
The distribution function of $\bZ$ is obtained from these joint survival functions by inclusion--exclusion. Recall that
\[
 B_{\Lambda_i}
 =\bigcup_{\substack{S\subset\mathbb I_d\\|S|=i}}
 \{\bz\in\R_+^d:z_j>a_j,\ j\in S\}.
\]
Let $\star\in\{=,\propto\}$ denote either the equal-rate or the proportional-rate model. Then, from \citet[Proposition 2.11]{das:fasen:2025}, we have for all $i\in \mathbb I_d$,

\begin{enumerate}
\item[(i)] the equal-rate model gives $\bZ\in \MRV(\alpha_i^{=},b_i^{=},\mu_i^{=},\E_d^{(i)})$, and
\item[(ii)] the proportional-rate model gives $\bZ\in \MRV(\alpha_i^{\propto},b_i^{\propto},\mu_i^{\propto},\E_d^{(i)})$,
\end{enumerate}
where 
\begin{equation*}
 b_{i,\star}^{\leftarrow}(u):=u^{\alpha_i^\star} \quad \text{and} \quad \mu_i^\star(B_{\Lambda_i})
 =\sum_{S\subset\mathbb I_d,\,|S|=i}
 \overline{\mu}_S^\star\big(\ba_S\big),
\end{equation*}
for $\star\in \{=,\propto\}$.
Let $N_\lambda$ be a Poisson process of intensity $\lambda>0$, independent of the i.i.d. claim vectors $(\bZ^{(k)})_{k\in\N}$, and set
$\bL(t)=\sum_{k=1}^{N_\lambda(t)}\bZ^{(k)}$. For both $\star\in \{=,\propto\}$, the subcone indices satisfy $\alpha_1^{\star} < \alpha_2^{\star} < \ldots < \alpha_d^{\star}$ and the strict subadditivity condition \eqref{eq:strict-subadd} (cf. \citet[Example 5]{das:fasen:2027}). Thus Proposition~\ref{prop:levy-transfer} gives the factor $\lambda T$, and Theorem~\ref{thm:ruin} yields
\begin{equation*}
 \lim_{u\to\infty}b_{i,\star}^{\leftarrow}(u)\psi_u^{\bQ}(\Lambda_i,T)
 = \lambda T \mu_i^{\star}(B_{\Lambda_i}),
 \qquad i=1,\ldots,d.
\end{equation*}

For the grouped regions in \Cref{ex:ruin-geometries}(d), the same notation gives the ruin behaviors
\begin{align*}
 &\lim_{u\to\infty}b_{r,\star}^{\leftarrow}(u)\psi_u^{\bQ}(\Lambda_{\mathrm{grp}},T)
 =\lambda T\Big(\prod_{\ell=1}^r|S_\ell|\Big)
 \overline{\mu}_{\{1,\ldots,r\}}^{\star}\big((\overline{a}_1,\ldots,\overline{a}_r)\big),\\
 &\lim_{u\to\infty}b_{r,\star}^{\leftarrow}(u)\psi_u^{\bQ}(\Lambda_{\mathrm{one/each}},T)
 =\lambda T
 \sum_{j_1\in S_1,\ldots,j_r\in S_r}
 \overline{\mu}_{\{j_1,\ldots,j_r\}}^{\star}\big((a_{j_1},\ldots,a_{j_r})\big),\\
 &\lim_{u\to\infty}b_{s_*,\star}^{\leftarrow}(u)\psi_u^{\bQ}(\Lambda_{\mathrm{one\ group}},T)
 =\lambda T
 \sum_{\ell:\,|S_\ell|=s_*}\overline{\mu}_{S_\ell}^{\star}\big(\ba_{S_\ell}\big).
\end{align*}

\subsection{Ruin in a bipartite insurance network}\label{sec:bipartite-insurance-network}

Consider an insurance company whose exposures are represented by a bipartite network.
The $d$ business lines of the insurance company are the objects in the bipartite network, while
$r$ ring-fenced insurance portfolios, each with separately allocated capital and
premium income, are the agents. The insurance portfolios include the losses from the company's various business segments. An edge between object $j$ and agent $\ell$ indicates that the $\ell$th insurance portfolio is exposed
to claims from business line $j$. 
A business-line object may be linked to several portfolios, and a portfolio may be
exposed to several business lines. Thus, a large claim in one business line can affect
several portfolios and create dependence between the insurance portfolios. This is the agent--object framework of
\cite{kley:kluppelberg:reinert:2016,Behme2020ruin}. 
These portfolios are
distinct from the reinsurance portfolios considered in \Cref{ex:ruin-geometries}(d) where the portfolios were disjoint.

Let $(\bL(t))_{t\geq 0}=((L_1(t),\ldots,L_d(t))^\top)_{t\geq 0}$ be the cumulative claim process generated by
the $d$ business lines and let $\bA\in\R_+^{r\times d}$ be the weighted
adjacency matrix, with no trivial rows. The entry $A_{\ell j}$ can be thought of as the proportion of claims of business line $j$ allocated to portfolio $\ell$. Throughout this subsection, we take the network to be fixed; a
random network can also be handled but requires further assumptions. Since the aggregate
claim of portfolio $\ell$ at time $t$ is $\sum_{j=1}^d A_{\ell j}L_j(t)$, the $\R_+^r$-valued claim process of the insurance
portfolios is
\begin{equation*}
 \bL^{(P)}(t)=\bA\bL(t),
 \qquad t\geq 0.
\end{equation*}
The superscript $(P)$ denotes quantities at the portfolio layer throughout.

The insurance portfolios have their
own initial capital and premium income; only the claims are transmitted through $\bA$.  Specifically, let
$\ba^{(P)}=(a_1^{(P)},\ldots,a_r^{(P)})^\top\in(0,\infty)^r$, with
$\sum_{\ell=1}^r a_\ell^{(P)}=1$, be the capital-allocation vector at the portfolio layer
and let $\bp^{(P)}\in\R_+^r$ be the corresponding premium-rate vector. Here $u>0$
denotes the insurer's total initial capital. The joint \textit{portfolio risk process (reserve process)} is
\begin{equation*}
 \bR_u^{(P)}(t)
 =u\ba^{(P)}+t\bp^{(P)}-\bL^{(P)}(t)
 =u\ba^{(P)}+t\bp^{(P)}-\bA\bL(t),
 \qquad t\geq 0.
\end{equation*}
The regular-variation behavior of the insurance portfolios is governed by
Breiman-type results \cite{breiman:1965,basrak:davis:mikosch:2002b,janssen:drees:2016} through the regular variation of the business-line claims.
The corresponding subcone result used below is given in
\cite[Section~3.2]{das:fasen:kluppelberg:2022}. A lighter-tailed perturbation could be
added to the reserve process, but we omit it here to simplify the notation.

For an insolvency event $\Lambda^{(r)}\subset\mathbb{R}^r$, the corresponding
finite-time \textit{insurance portfolio ruin probability} over the horizon $T>0$ is
\begin{eqnarray*}
 \psi_u^{(P)}(\Lambda^{(r)},T)
 &=&\P\!\left(
 \bR_u^{(P)}(t)\in\Lambda^{(r)}
 \text{ for some }t\in[0,T]
 \right).
\end{eqnarray*}
For $m\in\mathbb I_r$, we consider the event that at least $m$ portfolios have negative reserves, represented by the insolvency set
\begin{eqnarray*}
    \Lambda_m^{(r)}=\{\by\in\R^r:\#\{\ell\in\mathbb I_r:y_\ell<0\}\geq m\},
\end{eqnarray*}
which is now a subset of $\R^r$ instead of $\R^d$ as before.
Similarly, the proposed framework can be adapted to study other forms of insolvency.

As before, the normalized portfolio-level claim region of $\Lambda_m^{(r)}$ is
\[
 B_{\Lambda_m}^{(r)}
 =\{\by\in\R_+^r:\#\{\ell\in\mathbb I_r:y_{\ell}>a_{\ell}^{(P)}\}\ge m\}
 \in\mathcal{B}(\bbO_r^{(m)}),
\]
and the corresponding business-line claim region is the pre-image
\begin{equation*}
 \bA^{-1}(B_{\Lambda_m}^{(r)})
 =\{\bx\in\R_+^d:\#\{\ell\in\mathbb I_r:(\bA\bx)_{\ell}>a_{\ell}^{(P)}\}\ge m\}.
\end{equation*}
To find the correct subcone depth, we define
for $S\subset\mathbb I_d$ the set of insurance portfolios exposed to the business lines in $S$ as
\[
 \Gamma_{\bA}(S)=\{\ell \in \mathbb I_r:A_{\ell j}>0\text{ for some }j \in S\},
\]
and then
\[
 i_m(\bA)=\min\{|S|:S\subset\mathbb I_d,\ |\Gamma_{\bA}(S)|\ge m\}.
\]
Here $i_m(\bA)\in \{1,\ldots,d\}$ is the smallest number of business lines
 whose sufficiently large values can affect at least $m$
insurance portfolios. It determines the ruin depth.

\begin{remark}
Note $\bbO_d^{(i_m(\bA))}$ is the smallest subcone among $\bbO_d^{(i)}, i\in \mathbb I_d$, on which the pre-image of $\bbO_r^{(m)}$ under $\bA$ can be charged, implying that $\bA^{-1}(B_{\Lambda_m}^{(r)})\subset \bbO_d^{(i_m(\bA))} $.
Indeed, if the positive support of $\bx$ is $S=\{j\in\mathbb{I}_d:x_{j}>0\}$, then the positive support of $\bA\bx$ is $\Gamma_{\bA}(S)$. Hence, $\bA\bx$ can have at least $m$ positive components only if $|S|\ge i_m(\bA)$. Conversely, by the definition of $i_m(\bA)$, there is a set $S$ with $|S|=i_m(\bA)$ and $|\Gamma_{\bA}(S)|\ge m$; any vector with positive coordinates on $S$ and zero coordinates outside $S$ is mapped by $\bA$ into $\bbO_r^{(m)}$. 
\end{remark}

Having introduced all required notation, we can now formulate the asymptotic behavior of the portfolio ruin probability.

\begin{proposition}\label{prop:deterministic-matrix}
Assume that the $d$-dimensional insurance-line claim process $\bL$ satisfies Assumption~\ref{ass:riskmodel} 
and set $i:=i_m(\bA)$. If $\bA^{-1}(B_{\Lambda_m}^{(r)})$ is a $\mu_i$-continuity set and
$0<\mu_i(\bA^{-1}(B_{\Lambda_m}^{(r)}))<\infty$, then
\begin{equation*}
 \lim_{u\to\infty}b_i^{\leftarrow}(u)\psi_u^{(P)}(\Lambda_m^{(r)},T)
 =f_i(T)\mu_i\big(\bA^{-1}(B_{\Lambda_m}^{(r)})\big).
\end{equation*}
\end{proposition}

\begin{proof}
Following \citet[Theorem 3.4]{das:fasen:kluppelberg:2022}, we can map the regular variation of $\bL(T)$ to regular variation of $\bL^{(P)}(T)=\bA\bL(T)$ via
$$\{\bL^{(P)}(T)\in uB\}=\{\bL(T)\in u\,\bA^{-1}(B)\}.$$
In the notation used here, if $\bL(T)\in\MRV(\alpha_i,b_i,f_i(T)\mu_i,\bbO_d^{(i)})$ with $i=i_m(\bA)$ and if $B\in\mathcal{B}(\bbO_r^{(m)})$ is bounded away from $\CA_r^{(m-1)}$ with $\bA^{-1}(B)$ a $\mu_i$-continuity set, then
\begin{equation}\label{eq:matrix-rv-transfer}
 \lim_{u\to\infty}b_i^{\leftarrow}(u)\P\!\left(\bL^{(P)}(T)\in uB\right)= f_i(T)\mu_i(\bA^{-1}(B)).
\end{equation}
Since $\bA$ is nonnegative, $\bL^{(P)}=\bA\bL$ is an $r$-dimensional subordinator, and \eqref{eq:matrix-rv-transfer} supplies its required regular-variation limit with time factor $f_i(T)$.
Then \eqref{eq:matrix-rv-transfer} with $B=B_{\Lambda_m}^{(r)}$ identifies $i_m(\bA)$ as the relevant risk-object subcone depth and  Theorem~\ref{thm:ruin} gives the finite-time ruin limit.
\end{proof}

We discuss three useful cases next. Here we fix a matrix
$\bA\in\R_+^{r\times d}$ with no trivial rows and let $i=i_m(\bA)$ for some
$m\in\mathbb I_r$, where our interest is in the ruin set $\Lambda_m^{(r)}$.

\begin{enumerate}[(i)]
\item \textit{Independent business lines.}
Suppose that $L_1,\ldots,L_d$ are independent subordinators with L\'evy tail measure $\Pi_j((u,\infty))\sim c_ju^{-\alpha}$. Let $\nu_\alpha$ be the measure on $(0,\infty)$ satisfying $\nu_\alpha((x,\infty))=x^{-\alpha}$, and write $\bx_S^\circ$ for the vector in $\R_+^d$ obtained from $\bx_S$ by setting all coordinates outside $S\subset\{1,\ldots,d\}$ equal to zero. Then
\begin{equation*}
 \lim_{u\to\infty}u^{i\alpha}\psi_u^{(P)}(\Lambda_m^{(r)},T)
    =T^i C_m^{\rm ind}(\bA,\ba^{(P)}),
\end{equation*}
where
\begin{equation*}
 C_m^{\rm ind}(\bA,\ba^{(P)})
 =\sum_{\substack{S\subset\mathbb I_d\\|S|=i}}\Big(\prod_{j \in S}c_j\Big)
 \nu_\alpha^{\otimes S}\big(\{\bx_S:\bx_S^\circ\in \bA^{-1}(B_{\Lambda_m}^{(r)})\}\big).
\end{equation*}
Here $\nu_\alpha^{\otimes S}$ is the appropriate $|S|$-dimensional product measure. No separate Poisson rate appears here, since any rate parameter is already included in the marginal L\'evy-tail constants $c_j$.
In particular, when $i=1$ this reduces to
\begin{equation}\label{eq:cmind}
 \begin{aligned}
 C_m^{\rm ind}(\bA,\ba^{(P)})
 &=\sum_{j=1}^d c_j q_{j,m}^{-\alpha}  \qquad \text{ with } \\
 q_{j,m}
 &:=\text{the $m$th (increasing) order statistic of }
 \big\{a_\ell^{(P)}/A_{\ell j}:\ell\in\mathbb I_r,\ A_{\ell j}>0\big\},
 \end{aligned}
\end{equation}
and $q_{j,m}=\infty$ if business line $j$ affects fewer than $m$ insurance portfolios.

\item \textit{Gaussian-copula business-line claims.}
Let $\bL$ be a compound Poisson business-line claim process with Pareto$(\alpha)$ jumps and a $d$-dimensional Gaussian copula as in Section~\ref{sec:gaussian}, with $\rho>0$. Then
\begin{equation*}
 \lim_{u\to\infty}b_{i,\rho}^{\leftarrow}(u)\psi_u^{(P)}(\Lambda_m^{(r)},T)
    =\lambda T\,\mu_{i,\rho}^{(d)}\big(\bA^{-1}(B_{\Lambda_m}^{(r)})\big),
\end{equation*}
where $\mu_{i,\rho}^{(d)}$ is the $d$-dimensional Gaussian-copula subcone limit measure. For $i=1$, interpret $b_{1,\rho}^{\leftarrow}(u)=u^\alpha$ and $\mu_{1,\rho}^{(d)}$ as the limit measure in $\E_d^{(1)}$ given by
\[
 \mu_{1,\rho}^{(d)}\big(\bA^{-1}(B_{\Lambda_m}^{(r)})\big)
 =\sum_{j=1}^d q_{j,m}^{-\alpha},
\]
with $q_{j,m}$ as in \eqref{eq:cmind}.%; with marginal tail constants $c_r$, this sum becomes $\sum_{r=1}^d c_rq_{r,m}^{-\alpha}$.

\item \textit{Marshall--Olkin business-line claims.}
If the business-line claims have either of the $d$-dimensional Marshall--Olkin specifications from Section~\ref{sec:mo}, then for $\star\in\{=,\propto\}$,
\begin{equation*}
 \lim_{u\to\infty}b_{i,\star}^{\leftarrow}(u)\psi_u^{(P)}(\Lambda_m^{(r)},T)
    =\lambda T\,\mu_{i}^{\star, (d)}\big(\bA^{-1}(B_{\Lambda_m}^{(r)})\big).
\end{equation*}
\end{enumerate}

\begin{example}
As a concrete example, take $d=3$, $r=5$, $\lambda=1$ in the
compound-Poisson models, use the equal portfolio-capital
allocation $\ba^{(P)}=(1/5,\ldots,1/5)$,  and set
\[
 \bA=
 \begin{pmatrix}
 1&0&0\\
 0&1&0\\
 0&0&1\\
 1&1&0\\
 0&1&1
 \end{pmatrix}.
\]
The first, second, and third business lines respectively affect two, three, and two insurance portfolios. Therefore,
\[
 i_1(\bA)=i_2(\bA)=i_3(\bA)=1,
 \qquad
 i_4(\bA)=2,
 \qquad
 i_5(\bA)=3.
\]
\begin{itemize}
    \item[(i)] For independent business-line claim processes, the pre-images of
$B_{\Lambda_1}^{(5)}$ and $B_{\Lambda_2}^{(5)}$ receive contributions from all
three coordinate axes in $\E_d^{(1)}$ above the threshold $1/5$, whereas only the
second coordinate axis contributes to the pre-image of $B_{\Lambda_3}^{(5)}$.
The pre-image of $B_{\Lambda_4}^{(5)}$ requires both coordinates on each
two-coordinate face in $\E_3^{(2)}$ to exceed $1/5$, and that of
$B_{\Lambda_5}^{(5)}$ requires all three
coordinates to exceed $1/5$. Thus
\begin{align*}
 \lim_{u\to\infty}u^\alpha\psi_u^{(P)}(\Lambda_m^{(5)},T)
 &=T\,5^\alpha(c_1+c_2+c_3), \qquad (m=1,2),\\
 \lim_{u\to\infty}u^\alpha\psi_u^{(P)}(\Lambda_3^{(5)},T)
 &=T\,5^\alpha c_2,\\
 \lim_{u\to\infty}u^{2\alpha}\psi_u^{(P)}(\Lambda_4^{(5)},T)
 &=T^2\,5^{2\alpha}(c_1c_2+c_1c_3+c_2c_3),\\
 \lim_{u\to\infty}u^{3\alpha}\psi_u^{(P)}(\Lambda_5^{(5)},T)
 &=T^3\,5^{3\alpha}c_1c_2c_3.
\end{align*}
\item[(ii)] For both the Gaussian-copula and Marshall--Olkin models, the limits in $\E_3^{(1)}$ are
\begin{equation*}
 \lim_{u\to\infty}u^\alpha\psi_u^{(P)}(\Lambda_m^{(5)},T)
 =
 \begin{cases}
  T\,3\cdot5^\alpha, & m=1,2,\\
  T\,5^\alpha, & m=3.
 \end{cases}
\end{equation*}
On every two-coordinate face of $\E_{3}^{(2)}$, the pre-image of $B_{\Lambda_4}^{(5)}$ requires both active coordinates to exceed $1/5$, while the pre-image of $B_{\Lambda_5}^{(5)}$ on $\E_3^{(3)}$ requires all three coordinates to exceed $1/5$. Therefore, for the Gaussian-copula model with $\rho>0$,
\begin{align*}
 \lim_{u\to\infty}b_{2,\rho}^{\leftarrow}(u)\psi_u^{(P)}(\Lambda_4^{(5)},T)
 &=T\,3\Upsilon_{2,\rho}\,5^{2\alpha h_{2,\rho}},\\
 \lim_{u\to\infty}b_{3,\rho}^{\leftarrow}(u)\psi_u^{(P)}(\Lambda_5^{(5)},T)
 &=T\,\Upsilon_{3,\rho}\,5^{3\alpha h_{3,\rho}}.
\end{align*}
\item[(iii)] For the two Marshall--Olkin models, again with $\lambda=1$, the corresponding limits are
\begin{align*}
 \lim_{u\to\infty}u^{\alpha_2^\star}\psi_u^{(P)}(\Lambda_4^{(5)},T)
 &=T\,3\cdot5^{\alpha_2^\star},\\
 \lim_{u\to\infty}u^{\alpha_3^\star}\psi_u^{(P)}(\Lambda_5^{(5)},T)
 &=T\,5^{\alpha_3^\star},
 \qquad \star\in\{=,\propto\}.
\end{align*}
\end{itemize}
\end{example}

\FloatBarrier

\section{Numerical assessment}\label{sec:simulation} %\marginpar{\VF{I have only made minor changes}}
We examine two representative models from Section~\ref{sec:examples}: 
\begin{enumerate}
\item[(1)] Independent lines with unequal arrival rates, and
\item[(2)] Common arrivals with equi-correlated Gaussian-copula claims. 
\end{enumerate}
The numerical study assesses the finite-level accuracy of the asymptotic ruin probability and the solvency capital approximation (cf. \Cref{remark:solvency}). We use Monte Carlo simulation to quantify their errors; see \citet{asmussen:soren:glynn:2007} for details on stochastic simulation. Throughout, $$\psi_u^{\text{(A)}}:=f_i(T)\mu_i(B_{\Lambda_i})/b_i^{\leftarrow}(u)$$ denotes the relevant ruin approximation to $\psi_u(\Lambda_i,T)$ from \Cref{thm:ruin}, and $u_{p_{\mathrm{ruin}}}^{\text{(A)}}$ denotes the corresponding approximate solvency-capital requirement obtained by solving $\psi_u^{\text{(A)}}=p_{\mathrm{ruin}}$ for $u$ under the given allocation.
We also assess the ordering of the ruin probabilities under different allocation vectors $\ba$. In particular, if the tail formula gives $K(C,T,\ba)<K(C,T,\widetilde{\ba})$, we check whether our simulation also gives a smaller ruin probability under $\ba$ than under $\widetilde{\ba}$.

The simulations and numerical calculations were run in base R 4.5.2. The R
code used to produce
Tables~\ref{tab:ind-ruin}--\ref{tab:gau-capital} and
Figure~\ref{fig:numerical-checks} is available at
\url{https://github.com/bikram-jit-das/Multi-ruin-HT-finite}.
In both simulation studies, the number of business lines is $d=3$, the time horizon
is $T=1$, and the perturbation is $\bQ\equiv\bzero$. The ruin set is
$\Lambda_2$, so ruin requires at least two business lines to have negative
reserves at the same time. All marginal claim sizes have a common distribution
$F$ with $\overline F(z)=z^{-1.5}$ for $z\geq1$,
 which is a Pareto($\alpha$) distribution
with $\alpha=1.5$ and $\E Z=3$. 
We consider  premium rates $\bp^{(\theta)}=(p_1^{(\theta)},p_2^{(\theta)},p_3^{(\theta)})$  with \textit{safety loading} $\theta\in\{0.10,0.20,0.50\}$ and set
\begin{equation}\label{eq:simulation-premium}
 p_j^{(\theta)}=(1+\theta)\lambda_j\E Z,
 \qquad  j\in \{1,2,3\},
\end{equation}
where $\lambda_j$ is the marginal claim-arrival intensity, which is chosen differently in the two simulation studies. 
By \Cref{thm:ruin}, premiums do not affect the
ruin approximation as $u\to\infty$; the simulations below compare the effect of the three safety loadings (and hence premiums) relative to the theoretical findings.

\subsection{Independent lines with unequal arrival rates}

Consider $d=3$ independent claim processes, each following a compound Poisson process with an identical Pareto(1.5) claim distribution and with unequal Poisson intensities $\lambda_1=c_1=1$, $\lambda_2=c_2=2$, and $\lambda_3=c_3=4$.

Then using
\eqref{eq:independent-lines-limit} with $i=2$ and $T=1$, we obtain
\[
 \psi_u(\Lambda_2,1)
 \sim \psi_u^{\text{(A)}}
 :=K_2^{\rm ind}(\ba,1)u^{-3}
 =\sum_{1\leq j<k\leq3}c_jc_k(a_ja_k)^{-1.5}u^{-3}, \qquad \ u\to\infty.
\]
Numerical minimisation of this strictly convex objective $K_2^{\rm ind}(\ba,1)$ over $\ba\in \Delta_d^\circ$ gives the optimum
\begin{equation*}
 \ba^*=(0.26780,0.33648,0.39572)^\top
 \qquad \text{ with } \qquad K_2^{\rm ind}(\ba^*,1)=354.5398.
\end{equation*}
On the other hand, for the equal allocation vector
${\ba^{\text{eq}}}=(1/3,1/3,1/3)^\top$, we have
$K_2^{\rm ind}(\ba^{\text{eq}},1)=378$. Thus, relative to the equal allocation ${\ba^{\text{eq}}}$,
the optimal allocation $\ba^*$ reduces the value of $K_2^{\rm ind}(\ba,1)$, and hence the
ruin approximation at any fixed $u$, by
$1-354.5398/378=6.21\%$. Since
$u_{p_{\mathrm{ruin}}}^{\text{(A)}}=
\{K_2^{\rm ind}(\ba,1)/p_{\mathrm{ruin}}\}^{1/3}$,
the corresponding solvency capital reduction is
 $1-\left({354.5398}/{378}\right)^{1/3}=2.113\%.$

\begin{table}[t]
\caption{Independent lines: Monte Carlo ruin estimates  $\widehat\psi_u$  and ratios to
the ruin approximation $\psi_u^{\text{(A)}}$ for three safety loadings $\theta=0.10,\,0.20$ and $0.50$.  Binomial Monte Carlo standard errors (s.e.) are reported.}
\label{tab:ind-ruin}
\centering
\scriptsize
\setlength{\tabcolsep}{5pt}
\renewcommand{\arraystretch}{1.20}
\begin{tabular}{@{}cc|ccc|ccc@{}}
\toprule
&& \multicolumn{3}{|c}{Optimal allocation}
 & \multicolumn{3}{|c}{Equal allocation}\\
\cmidrule(lr){3-5}\cmidrule(lr){6-8}
$\theta$ & $u$ & $\widehat\psi_u$ & s.e.
 & $\widehat\psi_u/\psi_u^{\text{(A)}}$
 & $\widehat\psi_u$ & s.e. & $\widehat\psi_u/\psi_u^{\text{(A)}}$\\
\midrule
$0.10$ & 40  & $5.2236\times10^{-3}$ & $2.28\times10^{-5}$ & 0.943
                  & $5.4096\times10^{-3}$ & $2.32\times10^{-5}$ & 0.916\\
$0.10$ & 80  & $7.2580\times10^{-4}$ & $8.52\times10^{-6}$ & 1.048
                  & $7.7530\times10^{-4}$ & $8.80\times10^{-6}$ & 1.050\\
$0.10$ & 160 & $9.2000\times10^{-5}$ & $3.03\times10^{-6}$ & 1.063
                  & $9.8500\times10^{-5}$ & $3.14\times10^{-6}$ & 1.067\\
$0.10$ & 320 & $1.2600\times10^{-5}$ & $1.12\times10^{-6}$ & 1.165
                  & $1.3900\times10^{-5}$ & $1.18\times10^{-6}$ & 1.205\\
$0.10$ & 500 & $3.8000\times10^{-6}$ & $6.16\times10^{-7}$ & 1.340
                  & $4.5000\times10^{-6}$ & $6.71\times10^{-7}$ & 1.488\\
\addlinespace
\specialrule{0.1pt}{0pt}{0pt}
\addlinespace
$0.20$ & 40  & $4.7526\times10^{-3}$ & $2.17\times10^{-5}$ & 0.858
                  & $4.8782\times10^{-3}$ & $2.20\times10^{-5}$ & 0.826\\
$0.20$ & 80  & $6.8130\times10^{-4}$ & $8.25\times10^{-6}$ & 0.984
                  & $7.2750\times10^{-4}$ & $8.53\times10^{-6}$ & 0.985\\
$0.20$ & 160 & $8.9400\times10^{-5}$ & $2.99\times10^{-6}$ & 1.033
                  & $9.5400\times10^{-5}$ & $3.09\times10^{-6}$ & 1.034\\
$0.20$ & 320 & $1.2600\times10^{-5}$ & $1.12\times10^{-6}$ & 1.165
                  & $1.3600\times10^{-5}$ & $1.17\times10^{-6}$ & 1.179\\
$0.20$ & 500 & $3.8000\times10^{-6}$ & $6.16\times10^{-7}$ & 1.340
                  & $4.5000\times10^{-6}$ & $6.71\times10^{-7}$ & 1.488\\
\addlinespace
\specialrule{0.1pt}{0pt}{0pt}
\addlinespace
$0.50$ & 40  & $3.6666\times10^{-3}$ & $1.91\times10^{-5}$ & 0.662
                  & $3.6879\times10^{-3}$ & $1.92\times10^{-5}$ & 0.624\\
$0.50$ & 80  & $5.8600\times10^{-4}$ & $7.65\times10^{-6}$ & 0.846
                  & $6.1030\times10^{-4}$ & $7.81\times10^{-6}$ & 0.827\\
$0.50$ & 160 & $8.2500\times10^{-5}$ & $2.87\times10^{-6}$ & 0.953
                  & $8.7500\times10^{-5}$ & $2.96\times10^{-6}$ & 0.948\\
$0.50$ & 320 & $1.1600\times10^{-5}$ & $1.08\times10^{-6}$ & 1.072
                  & $1.2800\times10^{-5}$ & $1.13\times10^{-6}$ & 1.110\\
$0.50$ & 500 & $3.7000\times10^{-6}$ & $6.08\times10^{-7}$ & 1.305
                  & $4.3000\times10^{-6}$ & $6.56\times10^{-7}$ & 1.422\\
\bottomrule
\end{tabular}
\end{table}

We simulate $10{,}000{,}000$ independent paths as follows. For each path, we
first draw the total number of arrivals $N\sim\operatorname{Poisson}(7)$ and,
conditional on $N$, generate the ordered arrival times as the order
statistics of $N$ independent $\operatorname{Uniform}(0,1)$ random variables. Each
arrival is independently assigned to line $j$ with probability
$\lambda_j/7$, and its claim size is drawn independently from the
Pareto\,(1.5) distribution. This construction generates three
independent compound-Poisson claim processes with Poisson intensities
$(\lambda_1,\lambda_2,\lambda_3)=(1,2,4)$. % We then arrange the arrivals in chronological order, update the
%cumulative claims of the assigned line, and 
For the safety loadings $\theta=0.10,\,0.20$ and $0.50$, \eqref{eq:simulation-premium} results in the premium vectors 
$$\bp^{(0.10)}=(3.3,6.6,13.2)^\top, \quad
\bp^{(0.20)}=(3.6,7.2,14.4)^\top \quad \text{ and }  \quad
\bp^{(0.50)}=(4.5,9,18)^\top.$$
After each arrival
time $\tau_k$ we check whether
\[
 L_j(\tau_k)-p_j^{(\theta)}\tau_k>ua_j
\]
holds for at least two lines. If so, that path is recorded as ruined. Because
the reserves increase between arrivals, no new ruin event can occur between
the arrival times. The same claim paths are used for both allocations
and all three values of $\theta$.

\begin{table}[t]
\caption{Independent lines: Solvency-capital approximation
$u_{p_{\mathrm{ruin}}}^{\text{(A)}}$, estimated ruin probability
$\widehat\psi_{u_{p_{\mathrm{ruin}}}^{\text{(A)}}}$ achieved at that capital, and
direct Monte Carlo capital estimates $\widehat u_{p_{\mathrm{ruin}}}^{\text{(MC)}}$
for three safety loadings. The achieved-probability columns give binomial
standard errors in parentheses; the capital columns give
95\% CIs in brackets.}\label{tab:ind-capital}
\centering
\scriptsize
\setlength{\tabcolsep}{5pt}
\renewcommand{\arraystretch}{1.20}
\begin{tabular}{@{}cc|ccc|ccc@{}}
\toprule
&& \multicolumn{3}{|c}{Optimal allocation}
 & \multicolumn{3}{|c}{Equal allocation}\\
\cmidrule(lr){3-5}\cmidrule(lr){6-8}
$\theta$ & $p_{\mathrm{ruin}}$ & $u_{p_{\mathrm{ruin}}}^{\text{(A)}}$
 & $\widehat\psi_{u_{p_{\mathrm{ruin}}}^{\text{(A)}}}$
 & $\widehat u_{p_{\mathrm{ruin}}}^{\text{(MC)}}$
 & $u_{p_{\mathrm{ruin}}}^{\text{(A)}}$
 & $\widehat\psi_{u_{p_{\mathrm{ruin}}}^{\text{(A)}}}$
 & $\widehat u_{p_{\mathrm{ruin}}}^{\text{(MC)}}$\\[-1pt]
 & & & (s.e.) & [95\% CI] & & (s.e.) & [95\% CI]\\
\midrule
$0.10$ & $10^{-2}$ & 32.85 & $8.7457\times10^{-3}$ & 31.18
                    & 33.56 & $8.5210\times10^{-3}$ & 31.46\\[-1pt]
       &           &       & $(2.94\times10^{-5})$ & $[31.10,31.26]$
                    &       & $(2.91\times10^{-5})$ & $[31.39,31.54]$\\
$0.10$ & $10^{-3}$ & 70.78 & $1.0385\times10^{-3}$ & 71.64
                    & 72.30 & $1.0348\times10^{-3}$ & 73.10\\[-1pt]
       &           &       & $(1.02\times10^{-5})$ & $[71.18,72.12]$
                    &       & $(1.02\times10^{-5})$ & $[72.66,73.63]$\\
$0.10$ & $10^{-4}$ & 152.48 & $1.0680\times10^{-4}$ & 155.63
                    & 155.77 & $1.0590\times10^{-4}$ & 159.05\\[-1pt]
       &           &        & $(3.27\times10^{-6})$ & $[152.75,158.95]$
                    &        & $(3.25\times10^{-6})$ & $[155.73,162.46]$\\
\addlinespace
\specialrule{0.1pt}{0pt}{0pt}
\addlinespace
$0.20$ & $10^{-2}$ & 32.85 & $7.8360\times10^{-3}$ & 29.78
                    & 33.56 & $7.5880\times10^{-3}$ & 29.95\\[-1pt]
       &           &       & $(2.79\times10^{-5})$ & $[29.70,29.85]$
                    &       & $(2.74\times10^{-5})$ & $[29.87,30.02]$\\
$0.20$ & $10^{-3}$ & 70.78 & $9.7340\times10^{-4}$ & 70.16
                    & 72.30 & $9.6780\times10^{-4}$ & 71.44\\[-1pt]
       &           &       & $(9.86\times10^{-6})$ & $[69.69,70.60]$
                    &       & $(9.83\times10^{-6})$ & $[71.02,71.95]$\\
$0.20$ & $10^{-4}$ & 152.48 & $1.0380\times10^{-4}$ & 153.79
                    & 155.77 & $1.0250\times10^{-4}$ & 157.16\\[-1pt]
       &           &        & $(3.22\times10^{-6})$ & $[151.01,157.36]$
                    &        & $(3.20\times10^{-6})$ & $[154.15,161.00]$\\
\addlinespace
\specialrule{0.1pt}{0pt}{0pt}
\addlinespace
$0.50$ & $10^{-2}$ & 32.85 & $5.8542\times10^{-3}$ & 25.89
                    & 33.56 & $5.5627\times10^{-3}$ & 25.76\\[-1pt]
       &           &       & $(2.41\times10^{-5})$ & $[25.82,25.96]$
                    &       & $(2.35\times10^{-5})$ & $[25.68,25.83]$\\
$0.50$ & $10^{-3}$ & 70.78 & $8.1710\times10^{-4}$ & 65.76
                    & 72.30 & $8.0450\times10^{-4}$ & 66.72\\[-1pt]
       &           &       & $(9.04\times10^{-6})$ & $[65.37,66.23]$
                    &       & $(8.97\times10^{-6})$ & $[66.22,67.21]$\\
$0.50$ & $10^{-4}$ & 152.48 & $9.4900\times10^{-5}$ & 149.99
                    & 155.77 & $9.3300\times10^{-5}$ & 152.04\\[-1pt]
       &           &        & $(3.08\times10^{-6})$ & $[146.64,152.90]$
                    &        & $(3.05\times10^{-6})$ & $[149.07,155.17]$\\
\bottomrule
\end{tabular}
\end{table}

In Table~\ref{tab:ind-ruin}, we report the Monte Carlo estimated ruin probability $\widehat\psi_u$ as the sample mean, its standard error (s.e.) 
$\{\widehat\psi_u(1-\widehat\psi_u)/10{,}000{,}000\}^{1/2}$
and  the approximation $\psi_u^{\text{(A)}}=K_2^{\rm ind}(\ba,1)u^{-3}$. Higher safety loadings $\theta$ lower the estimated ruin probability $\widehat\psi_u$
most visibly at the smaller initial capital levels $u$; but their effect weakens as the capital level $u$
increases, which is consistent with the theoretical finding that the premium has no influence on the asymptotic ruin probability. For every displayed $(\theta,u)$, the optimal allocation (left part of Table~\ref{tab:ind-ruin})
has a smaller estimated ruin probability $\widehat\psi_u$ than the equal allocation (right part of Table~\ref{tab:ind-ruin}), in agreement
with the ordering of $K_2^{\rm ind}$. For $\theta=0.10$ and $0.20$, the
ratios $\widehat\psi_u/\psi_u^{\text{(A)}}$ are already near one at intermediate initial capitals $u$; for $\theta=0.50$, they
are closest to one at $u=160$ and $320$. The high ratios above $1$ at $u=500$ should be
interpreted in light of a greater relative Monte Carlo uncertainty,
because only around 40 paths (among 10,000,000 simulated paths) result in ruin.

The same simulated paths also give a direct numerical estimate of the
solvency capital. For a fixed safety loading $\theta$ and allocation $\ba$,
define the random variable
\begin{eqnarray} \label{eq.sim}
 M^{(\theta,\ba)}
 :=\max_{\tau_k\leq T}
 \,\,\,\min_{\{j,\ell\}\subset\{1,2,3\}}
 \left\{\frac{L_j(\tau_k)-p_j^{(\theta)}\tau_k}{a_j},\frac{L_{\ell}(\tau_k)-p_\ell^{(\theta)}\tau_k}{a_{\ell}}\right\}.
\end{eqnarray}
Given an initial capital $u$, at least two business lines become insolvent if and only if \linebreak $M^{(\theta,\ba)} > u$. Hence, the solvency capital associated with a ruin probability $p_{\mathrm{ruin}}$ is the $(1-p_{\mathrm{ruin}})$-quantile of the distribution of $M^{(\theta,\ba)}$. In practice, this  helps us estimate the solvency capital. For each of our $n$ simulations, we calculate \eqref{eq.sim}, resulting in $M_{1}^{(\theta,\ba)},\ldots,M_{n}^{(\theta,\ba)}$ and order them
in decreasing order so that  $M_{(1)}^{(\theta,\ba)}\geq\cdots\geq M_{(n)}^{(\theta,\ba)}$. Then we estimate the solvency capital for a given ruin probability $p_{\mathrm{ruin}}$ as
\begin{equation}\label{eq:mc-capital-estimator}
 \widehat u_{p_{\mathrm{ruin}}}^{\text{(MC)}}
 :=M_{(\lfloor n p_{\mathrm{ruin}}\rfloor+1)}^{(\theta,\ba)}.
\end{equation}
The 95\% confidence interval for the solvency capital is formed from two
order statistics whose ranks are determined by the 2.5\% and 97.5\%
quantiles of a $\operatorname{Binomial}(n,p_{\mathrm{ruin}})$ distribution.
Hence, this construction of a confidence interval is distribution-free.

A comparison of the approximate solvency capital
$u^{\text{(A)}}_{p_{\mathrm{ruin}}}$ obtained by solving \linebreak
$K_2^{\rm ind}(\ba,1)u^{-3}=p_{\mathrm{ruin}}$ and  the
estimate $\widehat u^{\text{(MC)}}_{p_{\mathrm{ruin}}}$ of
\eqref{eq:mc-capital-estimator} is presented in Table~\ref{tab:ind-capital}. Since
$K_2^{\rm ind}(\ba,1)$ is independent of the premium vector,
$u_{p_{\mathrm{ruin}}}^{\text{(A)}}$ is the same for all three values of $\theta$,
whereas the Monte Carlo estimated solvency capital $\widehat u_{p_{\mathrm{ruin}}}^{\text{(MC)}}$ falls as the safety loading
increases. An achieved probability
$\widehat\psi_{u_{p_{\mathrm{ruin}}}^{\text{(A)}}}$ above the target means that the
approximation provides too little capital; a value below the target means
that it is conservative. At $p_{\mathrm{ruin}}=10^{-4}$, the Monte Carlo 
estimated capital reductions from the optimal rather than equal allocation are
$2.15\%$, $2.14\%$, and $1.34\%$ for $\theta=0.10$, $0.20$, and $0.50$,
respectively, which is close to the asymptotic reduction of $2.11\%$.
At least for the pair $(\theta,p_{\mathrm{ruin}})=(0.50,10^{-2})$, the
ordering of the estimated solvency capital illustrates that
an asymptotically optimal allocation need not be  optimal (in a simulated scenario)  when
premium effects remain substantial.

\subsection{Gaussian-copula claims}

We next use the common-arrival Gaussian-copula model from
Section~\ref{sec:gaussian}, with arrival intensity $\lambda=1$, equicorrelation
parameter $\rho=0.5$, and Pareto\,(1.5) marginal claim sizes. Here
$\gamma_{2,0.5}=4/3$, $h_{2,0.5}=2/3$ and $\alpha h_{2,0.5}=1$. From
\eqref{eq:gaussian-ruin-limit}, we have
\[
 K_2^{\rm Gau}(\ba,1)
 :=\Upsilon_{2,0.5}\sum_{1\leq j<k\leq3}(a_ja_k)^{-1} \qquad
 \text{ with }
 \qquad \Upsilon_{2,0.5}\approx0.4134967.
\]
Over $\Delta_3^\circ$, the equal allocation
$\ba^{\text{eq}}=(1/3,1/3,1/3)^\top$ minimizes
$K_2^{\rm Gau}(\ba,1)$, so that \linebreak $\ba^*=\ba^{\text{eq}}$. We contrast the equal allocation with the
deliberately skewed allocation \linebreak $\ba^{\text{sk}}=(0.2,0.3,0.5)^\top$. The
corresponding constants are
\[
 K_2^{\rm Gau}(\ba^{\text{eq}},1)=11.1644
 \qquad \text{ and } \qquad
 K_2^{\rm Gau}(\ba^{\text{sk}},1)=13.7832.
\]
The skewed allocation therefore increases the value of $K_2^{\rm Gau}(\ba,1)$ by $23.46\%$. \linebreak
Since $b_{2,0.5}^{\leftarrow}\in\RV_2$, its generalized inverse satisfies
$b_{2,0.5}\in\RV_{1/2}$. Thus, as $p_{\mathrm{ruin}}\downarrow0$,
\[
 \frac{u_{p_{\mathrm{ruin}}}^{\text{(A)}}(\ba^{\mathrm{eq}})}
      {u_{p_{\mathrm{ruin}}}^{\text{(A)}}(\ba^{\mathrm{sk}})}
 \longrightarrow
 \left\{
 \frac{K_2^{\rm Gau}(\ba^{\mathrm{eq}},1)}
      {K_2^{\rm Gau}(\ba^{\mathrm{sk}},1)}
 \right\}^{1/2}=0.9.
\]
Consequently, relative to the skewed allocation, the asymptotic capital
reduction under the optimal equal allocation is $\left(1-
 \{{K_2^{\rm Gau}(\ba^{\mathrm{eq}},1)}/
      K_2^{\rm Gau}(\ba^{\mathrm{sk}},1)\}
 ^{1/2}\right)100\%= 10.00\%.$

As before, we simulate $n=10{,}000{,}000$ compound-Poisson paths. For each path,
we draw $N\sim\operatorname{Poisson}(1)$ and, conditional on $N$, generate the
ordered common arrival times as the order statistics of $N$ independent
$\operatorname{Uniform}(0,1)$ random variables. At each arrival, we draw a claim
vector with a Gaussian copula having equicorrelation parameter $\rho=0.5$ and
Pareto\,(1.5) margins. We then follow the same procedure as in the independent-line model with safety loadings $\theta=0.10,\,0.20,\,0.50$
and premium vectors
$$\bp^{(0.10)}=(3.3,3.3,3.3)^\top, \quad
\bp^{(0.20)}=(3.6,3.6,3.6)^\top \quad \text{ and } \quad
\bp^{(0.50)}=(4.5,4.5,4.5)^\top.$$
\begin{table}[t]
\caption{Gaussian-copula claims: Monte Carlo ruin estimates  $\widehat\psi_u$  and ratios to
the ruin approximation $\psi_u^{\text{(A)}}$ for three safety loadings $\theta=0.10,\,0.20$ and $0.50$.  Binomial Monte Carlo standard errors (s.e.) are reported.}\label{tab:gau-ruin}
\centering
\scriptsize
\setlength{\tabcolsep}{5pt}
\renewcommand{\arraystretch}{1.20}
\begin{tabular}{@{}cc|ccc|ccc@{}}
\toprule
&& \multicolumn{3}{|c}{Equal allocation}
 & \multicolumn{3}{|c}{Skewed allocation}\\
\cmidrule(lr){3-5}\cmidrule(lr){6-8}
$\theta$ & $u$ & $\widehat\psi_u$ & s.e. & $\widehat\psi_u/\psi_u^{\text{(A)}}$
    & $\widehat\psi_u$ & s.e. & $\widehat\psi_u/\psi_u^{\text{(A)}}$\\
\midrule
$0.10$ & 40  & $1.0483\times10^{-2}$ & $3.22\times10^{-5}$ & 0.983
                  & $1.3207\times10^{-2}$ & $3.61\times10^{-5}$ & 1.003\\
$0.10$ & 80  & $2.3912\times10^{-3}$ & $1.54\times10^{-5}$ & 0.950
                  & $3.0269\times10^{-3}$ & $1.74\times10^{-5}$ & 0.974\\
$0.10$ & 160 & $5.5780\times10^{-4}$ & $7.47\times10^{-6}$ & 0.931
                  & $6.9730\times10^{-4}$ & $8.35\times10^{-6}$ & 0.943\\
$0.10$ & 320 & $1.3280\times10^{-4}$ & $3.64\times10^{-6}$ & 0.925
                  & $1.7020\times10^{-4}$ & $4.13\times10^{-6}$ & 0.961\\
$0.10$ & 500 & $5.2300\times10^{-5}$ & $2.29\times10^{-6}$ & 0.912
                  & $6.7400\times10^{-5}$ & $2.60\times10^{-6}$ & 0.952\\
\addlinespace
\specialrule{0.1pt}{0pt}{0pt}
\addlinespace
$0.20$ & 40  & $1.0187\times10^{-2}$ & $3.18\times10^{-5}$ & 0.955
                  & $1.2773\times10^{-2}$ & $3.55\times10^{-5}$ & 0.970\\
$0.20$ & 80  & $2.3577\times10^{-3}$ & $1.53\times10^{-5}$ & 0.937
                  & $2.9715\times10^{-3}$ & $1.72\times10^{-5}$ & 0.956\\
$0.20$ & 160 & $5.5480\times10^{-4}$ & $7.45\times10^{-6}$ & 0.926
                  & $6.9000\times10^{-4}$ & $8.30\times10^{-6}$ & 0.933\\
$0.20$ & 320 & $1.3210\times10^{-4}$ & $3.63\times10^{-6}$ & 0.920
                  & $1.6890\times10^{-4}$ & $4.11\times10^{-6}$ & 0.953\\
$0.20$ & 500 & $5.2200\times10^{-5}$ & $2.28\times10^{-6}$ & 0.910
                  & $6.7100\times10^{-5}$ & $2.59\times10^{-6}$ & 0.948\\
\addlinespace
\specialrule{0.1pt}{0pt}{0pt}
\addlinespace
$0.50$ & 40  & $9.3945\times10^{-3}$ & $3.05\times10^{-5}$ & 0.881
                  & $1.1585\times10^{-2}$ & $3.38\times10^{-5}$ & 0.880\\
$0.50$ & 80  & $2.2571\times10^{-3}$ & $1.50\times10^{-5}$ & 0.897
                  & $2.8244\times10^{-3}$ & $1.68\times10^{-5}$ & 0.909\\
$0.50$ & 160 & $5.4490\times10^{-4}$ & $7.38\times10^{-6}$ & 0.909
                  & $6.7270\times10^{-4}$ & $8.20\times10^{-6}$ & 0.909\\
$0.50$ & 320 & $1.3080\times10^{-4}$ & $3.62\times10^{-6}$ & 0.911
                  & $1.6690\times10^{-4}$ & $4.08\times10^{-6}$ & 0.942\\
$0.50$ & 500 & $5.1800\times10^{-5}$ & $2.28\times10^{-6}$ & 0.903
                  & $6.6200\times10^{-5}$ & $2.57\times10^{-6}$ & 0.935\\
\bottomrule
\end{tabular}
\end{table}

Table~\ref{tab:gau-ruin} provides the Monte Carlo estimated ruin probability $\widehat\psi_u$
with the approximation
$
 \psi_u^{\text{(A)}}=K_2^{\rm Gau}(\ba,1)/b_{2,0.5}^{\leftarrow}(u)
$
for both allocations ($\ba^{\mathrm{eq}}$ on the left-hand side and $\ba^{\mathrm{sk}}$ on the right-hand side of  Table~\ref{tab:gau-ruin}) and all three safety loadings with $\theta=0.1,0.2$ and $0.5$. In every scenario, the equal
allocation has a lower estimated ruin probability $\widehat \psi_u$ than the skewed allocation,
in agreement with the ordering of $K_2^{\rm Gau}(\ba,1)$. Increasing the premium
loading $\theta$ again lowers the estimated ruin probability $\widehat \psi_u$; however, for higher levels of $u$, the different safety loadings no longer have a discernible effect.

Additionally, Table~\ref{tab:gau-capital} shows for each target probability
$p_{\mathrm{ruin}}$ both the achieved probability at the approximate solvency
capital $u_{p_{\mathrm{ruin}}}^{\text{(A)}}$ and the direct Monte Carlo  estimated solvency capital $\widehat u_{p_{\mathrm{ruin}}}^{\text{(MC)}}$. Except for the skewed
allocation with $\theta=0.10$ and $p_{\mathrm{ruin}}=10^{-2}$, the achieved
probabilities are below their targets, so the Gaussian capital approximation
is generally conservative for the premium rates considered. The effect is
most pronounced for $\theta=0.50$ at $p_{\mathrm{ruin}}=10^{-2}$. More
importantly, the allocation comparison can now be made using the Monte Carlo estimated solvency capital rather than the two approximating formulas. Across the
nine displayed $(\theta,p_{\mathrm{ruin}})$ pairs, the estimated capital
reduction from equal rather than skewed allocation ranges from $9.51\%$ to
$11.19\%$, close to the asymptotic reduction of $10.00\%$. For example, at
$(\theta,p_{\mathrm{ruin}})=(0.20,10^{-4})$ the estimated reduction is
$(1-367.01/413.24)100\%=11.19\%$.

\begin{table}[t]
\caption{Gaussian-copula claims: Solvency-capital approximation
$u_{p_{\mathrm{ruin}}}^{\text{(A)}}$, estimated ruin probability
$\widehat\psi_{u_{p_{\mathrm{ruin}}}^{\text{(A)}}}$ achieved at that capital, and
direct Monte Carlo capital estimates $\widehat u_{p_{\mathrm{ruin}}}^{\text{(MC)}}$
for three safety loadings. The achieved-probability columns give binomial
standard errors in parentheses; the capital columns give
95\% CIs in brackets.}\label{tab:gau-capital}
\centering
\scriptsize
\setlength{\tabcolsep}{5pt}
\renewcommand{\arraystretch}{1.10}
\begin{tabular}{@{}cc|ccc|ccc@{}}
\toprule
&& \multicolumn{3}{|c}{Equal allocation}
 & \multicolumn{3}{|c}{Skewed allocation}\\
\cmidrule(lr){3-5}\cmidrule(lr){6-8}
$\theta$ & $p_{\mathrm{ruin}}$ & $u_{p_{\mathrm{ruin}}}^{\text{(A)}}$
 & $\widehat\psi_{u_{p_{\mathrm{ruin}}}^{\text{(A)}}}$
 & $\widehat u_{p_{\mathrm{ruin}}}^{\text{(MC)}}$
 & $u_{p_{\mathrm{ruin}}}^{\text{(A)}}$
 & $\widehat\psi_{u_{p_{\mathrm{ruin}}}^{\text{(A)}}}$
 & $\widehat u_{p_{\mathrm{ruin}}}^{\text{(MC)}}$\\[-1pt]
 & & & (s.e.) & [95\% CI] & & (s.e.) & [95\% CI]\\
\midrule
$0.10$ & $10^{-2}$ & 41.25 & $9.8286\times10^{-3}$ & 40.90
                    & 45.62 & $1.0040\times10^{-2}$ & 45.72\\[-1pt]
       &           &       & $(3.12\times10^{-5})$ & $[40.79,41.02]$
                    &       & $(3.15\times10^{-5})$ & $[45.58,45.86]$\\
$0.10$ & $10^{-3}$ & 124.88 & $9.3630\times10^{-4}$ & 120.87
                    & 138.28 & $9.3680\times10^{-4}$ & 133.94\\[-1pt]
       &           &        & $(9.67\times10^{-6})$ & $[119.71,122.03]$
                    &        & $(9.67\times10^{-6})$ & $[132.62,135.11]$\\
$0.10$ & $10^{-4}$ & 381.46 & $9.2100\times10^{-5}$ & 367.40
                    & 422.63 & $9.5900\times10^{-5}$ & 413.68\\[-1pt]
       &           &        & $(3.03\times10^{-6})$ & $[356.58,377.58]$
                    &        & $(3.10\times10^{-6})$ & $[398.29,427.28]$\\
\addlinespace
\specialrule{0.1pt}{0pt}{0pt}
\addlinespace
$0.20$ & $10^{-2}$ & 41.25 & $9.5670\times10^{-3}$ & 40.35
                    & 45.62 & $9.7437\times10^{-3}$ & 45.05\\[-1pt]
       &           &       & $(3.08\times10^{-5})$ & $[40.23,40.48]$
                    &       & $(3.11\times10^{-5})$ & $[44.91,45.19]$\\
$0.20$ & $10^{-3}$ & 124.88 & $9.2830\times10^{-4}$ & 120.35
                    & 138.28 & $9.2680\times10^{-4}$ & 133.30\\[-1pt]
       &           &        & $(9.63\times10^{-6})$ & $[119.15,121.53]$
                    &        & $(9.62\times10^{-6})$ & $[131.93,134.51]$\\
$0.20$ & $10^{-4}$ & 381.46 & $9.1800\times10^{-5}$ & 367.01
                    & 422.63 & $9.5500\times10^{-5}$ & 413.24\\[-1pt]
       &           &        & $(3.03\times10^{-6})$ & $[355.78,377.12]$
                    &        & $(3.09\times10^{-6})$ & $[397.73,426.67]$\\
\addlinespace
\specialrule{0.1pt}{0pt}{0pt}
\addlinespace
$0.50$ & $10^{-2}$ & 41.25 & $8.8389\times10^{-3}$ & 38.77
                    & 45.62 & $8.9251\times10^{-3}$ & 43.12\\[-1pt]
       &           &       & $(2.96\times10^{-5})$ & $[38.64,38.89]$
                    &       & $(2.97\times10^{-5})$ & $[42.98,43.25]$\\
$0.50$ & $10^{-3}$ & 124.88 & $9.0480\times10^{-4}$ & 118.81
                    & 138.28 & $9.0280\times10^{-4}$ & 131.30\\[-1pt]
       &           &        & $(9.51\times10^{-6})$ & $[117.65,119.86]$
                    &        & $(9.50\times10^{-6})$ & $[130.10,132.76]$\\
$0.50$ & $10^{-4}$ & 381.46 & $9.0800\times10^{-5}$ & 365.25
                    & 422.63 & $9.5100\times10^{-5}$ & 411.22\\[-1pt]
       &           &        & $(3.01\times10^{-6})$ & $[353.98,375.69]$
                    &        & $(3.08\times10^{-6})$ & $[396.34,424.84]$\\
\bottomrule
\end{tabular}
\end{table}

\begin{figure}[t]
\centering
\includegraphics[width=\textwidth]{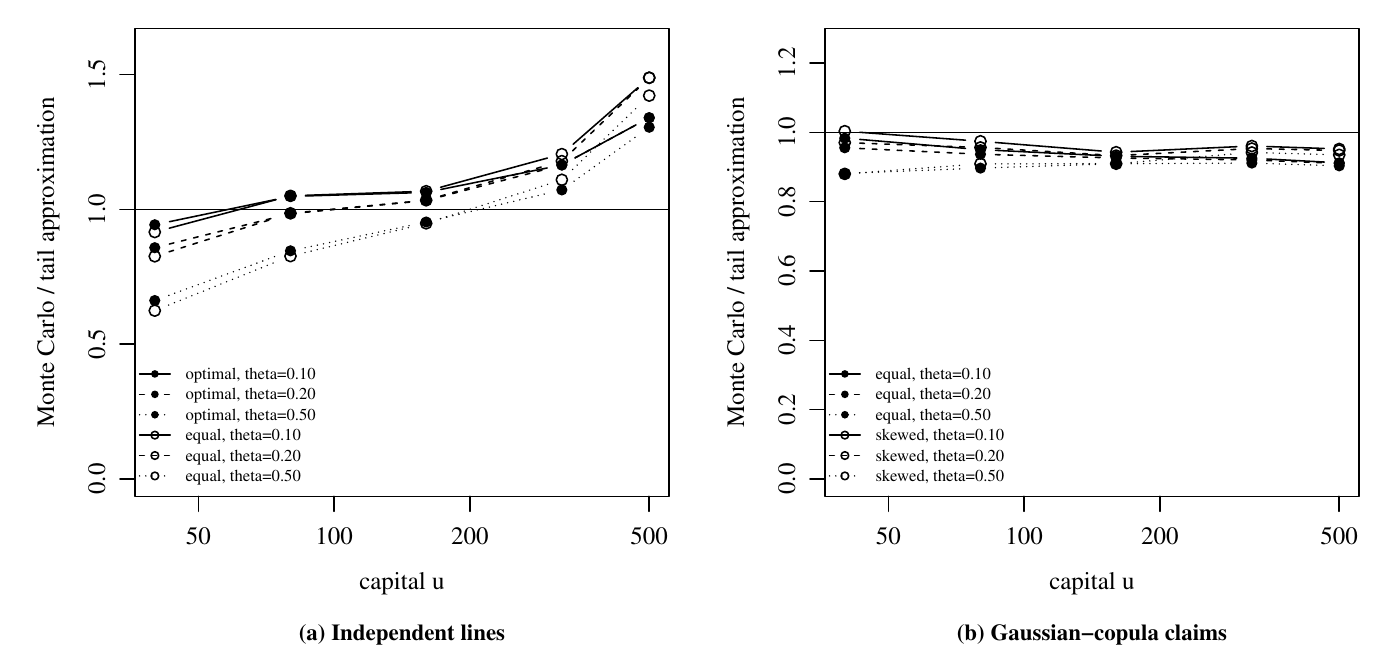}
\caption{Monte Carlo ruin estimates divided by their tail approximations $\widehat\psi_u/\psi_u^{\text{(A)}}$ for
positive safety loadings $\theta=0.10$, $0.20$, and $0.50$. Panel (a) shows independent lines
under the optimal and equal allocations, and panel (b) shows Gaussian-copula
claims under equal and skewed allocations. The horizontal line is at one.}
\label{fig:numerical-checks}

\end{figure}

Finally, Figure~\ref{fig:numerical-checks} summarizes the accuracy by plotting
the ratio $\widehat\psi_u/\psi_u^{(A)}$ of the estimated ruin probability to its approximation. Within
each allocation, the curves for the three safety loadings become closer as the initial
capital grows, illustrating that the premium has no asymptotic
effect. Across all scenarios investigated, the value of $K_2^{\rm Gau}(\ba,1)$ correctly
identifies which allocation has the smaller ruin probability at a common initial
capital level $u$, even where the absolute tail approximation is not perfect.
\FloatBarrier

\section{Conclusion}\label{sec:conclusion}

In this paper, we have demonstrated how the geometry of an insolvency region determines the particular tail region governing finite-horizon ruin and, consequently, its speed of decrease in a heavy tailed Lévy-driven insurance risk model. The limit probability separates the claim-arrival mechanism from the effects of extremal dependence and capital allocation. In contrast to the classical  one-large-jump formulas for regularly varying claim sizes, we notice that the time factor may be nonlinear and may depend on how large claims are shared across business lines and their dependence structure. As a caveat, one should note that the resulting formulas serve as tail benchmarks whose practical accuracy must be assessed numerically. 

Extending the theory to two-sided risk processes, investment returns, or infinite-horizon ruin remains an important direction for future work, since the sample paths for such processes can leave and re-enter favorable regions and the terminal values no longer capture the ruin event.  Such problems would require  large-deviation or renewal arguments rather than the finite-time comparison that is used here.

 Finally, statistical implementation of our results also requires joint estimation of the relevant subcone, its tail index, and its mass on the insolvency region, together with an assessment of how estimation uncertainty affects ruin probabilities and capital allocation. Data-driven selection of the tail layer and rare-event simulation adapted to the particular subcones would further improve finite-level performance.

\section*{Declaration of competing interests statement} The authors declare that they have no competing interests.

\section*{Declaration of generative AI use}
 During the preparation of this work the authors used Open-AI and Copilot for language editing, table formatting, computing certain examples from the proposed theoretical results, and constructing and executing R code. After using these tools, the authors reviewed and edited the content as needed and take full responsibility for the content of the published article.

\bibliographystyle{imsart-nameyear}
\bibliography{bibfilenew}

\appendix

\appendix

\section{Summary of model-dependent ruin rates}\label{app:summary-rates}

A summary of the ruin behaviors for the various models, dependence structures, and insolvency regions presented in this paper is provided 
in Table~\ref{tab:rate-summary}. For each dependence structure (first column) and insolvency set (second column), the probability-scale column (third column) lists only the corresponding function of $u$, the time column (fourth column) specifies the time horizon or arrival factor, and the final column contains the remaining constants related to the dependence structure and geometry. 

The key notation is as follows. The number $r$ is the number of reinsurance portfolios, $\overline{a}_\ell=\sum_{j\in S_\ell}a_j$, $s_*=\min_\ell |S_\ell|$, and for singleton portfolios write $S_\ell=\{j_\ell\}$. The quantities $\gamma_{i,\rho}$, $h_{i,\rho}$, and $\Upsilon_{i,\rho}$ are defined in Section~\ref{sec:gaussian}; $\star\in\{=,\propto\}$ and $\bar{\mu}_S^\star$ are as in Section~\ref{sec:mo}. For the bipartite-network rows, $i=i_m(\bA)$ and the geometry is expressed through $\bA^{-1}(B_{\Lambda_m}^{(r)})$, as in Section~\ref{sec:bipartite-insurance-network}.

\pagestyle{empty}

\begin{landscape}
\makeatletter
\PLS@Rotate{0}
\makeatother
\begin{table}[p]
\caption{For a model setting in column 1, a ruin region in column 2, and a large initial capital $u$, the corresponding finite-time ruin probability is approximated by the product of the terms in columns 3, 4, and 5. The probability is $\psi_u^{\bQ}(\cdot,T)$ for models (1)--(2) and $\psi_u^{(P)}(\cdot,T)$ for model (3). This table summarizes the examples in Section~\ref{sec:examples}.}\label{tab:rate-summary}
\centering
\tiny
\setlength{\tabcolsep}{2pt}
\renewcommand{\arraystretch}{1.3}
\begin{tabular}{@{}p{0.189\linewidth}p{0.13\linewidth}p{0.17\linewidth}p{0.10\linewidth}p{0.39\linewidth}@{}}
\toprule
Model & Ruin region & Probability scale & Time factor & Constant\\
\midrule
(1) Independent subordinators
& $\Lambda_m$
& $u^{-m\alpha}$
& $T^m$
& $\sum_{|S|=m}\prod_{j\in S}c_ja_j^{-\alpha}$\\
& $\Lambda_{\rm grp}$
& $u^{-r\alpha}$
& $T^r$
& $\prod_{\ell=1}^r\left(\overline{a}_\ell^{-\alpha}\sum_{j\in S_\ell}c_j\right)$\\
& $\Lambda_{\rm one/each}$
& $u^{-r\alpha}$
& $T^r$
& $\prod_{\ell=1}^r\left(\sum_{j\in S_\ell}c_ja_j^{-\alpha}\right)$\\
& $\Lambda_{\rm one\ group}$
& $u^{-s_*\alpha}$
& $T^{s_*}$
& $\sum_{\ell:\,|S_\ell|=s_*}\prod_{j\in S_\ell}c_ja_j^{-\alpha}$\\
\midrule
(2) Common arrivals & & & &\\
\midrule
 \; (a) Independent claims
& $\Lambda_m$
& $u^{-m\alpha}$
& $\E[N_\lambda(T)^m]$
& $\sum_{|S|=m}\prod_{j\in S}c_ja_j^{-\alpha}$\\
& $\Lambda_{\rm grp}$
& $u^{-r\alpha}$
& $\E[N_\lambda(T)^r]$
& $\prod_{\ell=1}^r\left(\overline{a}_\ell^{-\alpha}\sum_{j\in S_\ell}c_j\right)$\\
& $\Lambda_{\rm one/each}$
& $u^{-r\alpha}$
& $\E[N_\lambda(T)^r]$
& $\prod_{\ell=1}^r\left(\sum_{j\in S_\ell}c_ja_j^{-\alpha}\right)$\\
& $\Lambda_{\rm one\ group}$
& $u^{-s_*\alpha}$
& $\E[N_\lambda(T)^{s_*}]$
& $\sum_{\ell:\,|S_\ell|=s_*}\prod_{j\in S_\ell}c_ja_j^{-\alpha}$\\
\midrule
\; (b) Gaussian-copula claims, $\rho>0$
& $\Lambda_m$, $m\geq2$
& $(\log u)^{-(m-\gamma_{m,\rho})/2}u^{-\alpha\gamma_{m,\rho}}$
& $\lambda T$
& $(2\pi)^{\gamma_{m,\rho}/2}(2\alpha)^{-(m-\gamma_{m,\rho})/2}\Upsilon_{m,\rho}\sum_{|S|=m}\prod_{j\in S}a_j^{-\alpha h_{m,\rho}}$\\
& $\Lambda_{\rm grp}$, $r\geq2$
& $(\log u)^{-(r-\gamma_{r,\rho})/2}u^{-\alpha\gamma_{r,\rho}}$
& $\lambda T$
& $(2\pi)^{\gamma_{r,\rho}/2}(2\alpha)^{-(r-\gamma_{r,\rho})/2}\Upsilon_{r,\rho}\prod_{\ell=1}^r |S_\ell|\overline{a}_\ell^{-\alpha h_{r,\rho}}$\\
& $\Lambda_{\rm one/each}$, $r\geq2$
& $(\log u)^{-(r-\gamma_{r,\rho})/2}u^{-\alpha\gamma_{r,\rho}}$
& $\lambda T$
& $(2\pi)^{\gamma_{r,\rho}/2}(2\alpha)^{-(r-\gamma_{r,\rho})/2}\Upsilon_{r,\rho}\prod_{\ell=1}^r\sum_{j\in S_\ell}a_j^{-\alpha h_{r,\rho}}$\\
& $\Lambda_{\rm one\ group}$, $s_*=1$
& $u^{-\alpha}$
& $\lambda T$
& $\sum_{\ell:\,S_\ell=\{j_\ell\}}a_{j_\ell}^{-\alpha}$\\
& $\Lambda_{\rm one\ group}$, $s_*\ge2$
& $(\log u)^{-(s_*-\gamma_{s_*,\rho})/2}u^{-\alpha\gamma_{s_*,\rho}}$
& $\lambda T$
& $(2\pi)^{\gamma_{s_*,\rho}/2}(2\alpha)^{-(s_*-\gamma_{s_*,\rho})/2}\Upsilon_{s_*,\rho}\sum_{\ell:\,|S_\ell|=s_*}\prod_{j\in S_\ell}a_j^{-\alpha h_{s_*,\rho}}$\\
\midrule
\; (c) Marshall--Olkin copula claims
& $\Lambda_m$
& $u^{-\alpha_m^\star}$
& $\lambda T$
& $\mu_m^\star(B_{\Lambda_m})$\\
& $\Lambda_{\rm grp}$
& $u^{-\alpha_r^\star}$
& $\lambda T$
& $\mu_r^\star(B_{\Lambda_{\rm grp}})=\left(\prod_{\ell=1}^r|S_\ell|\right)\bar{\mu}_{\{1,\ldots,r\}}^\star((\overline{a}_1,\ldots,\overline{a}_r))$\\
& $\Lambda_{\rm one/each}$
& $u^{-\alpha_r^\star}$
& $\lambda T$
& $\mu_r^\star(B_{\Lambda_{\rm one/each}})=\sum_{j_1\in S_1,\ldots,j_r\in S_r}\bar{\mu}_{\{j_1,\ldots,j_r\}}^\star((a_{j_1},\ldots,a_{j_r}))$\\
& $\Lambda_{\rm one\ group}$
& $u^{-\alpha_{s_*}^\star}$
& $\lambda T$
& $\mu_{s_*}^\star(B_{\Lambda_{\rm one\ group}})=\sum_{\ell:\,|S_\ell|=s_*}\bar{\mu}_{S_\ell}^\star(\ba_{S_\ell})$\\
\midrule
(3) Bipartite insurance network, $\bL^{(P)}=\bA\bL$ & & & &\\
\midrule
\; (a) Independent business-line claims
& $\Lambda_m^{(r)}$, $i=i_m(\bA)$
& $u^{-i\alpha}$
& $T^i$
& $\sum_{\substack{S\subset\mathbb I_d\\|S|=i}}\left(\prod_{j\in S}c_j\right) \nu_\alpha^{\otimes S}\{\bx_S:\bx_S^\circ\in \bA^{-1}(B_{\Lambda_m}^{(r)})\}$\\
\; (b) Gaussian-copula business-line claims, $\rho>0$
& $\Lambda_m^{(r)}$, $i=i_m(\bA)=1$
& $u^{-\alpha}$
& $\lambda T$
& $\mu_{1,\rho}^{(d)}\big(\bA^{-1}(B_{\Lambda_m}^{(r)})\big)$\\
& $\Lambda_m^{(r)}$, $i=i_m(\bA)\ge2$
& $(\log u)^{-(i-\gamma_{i,\rho})/2}u^{-\alpha\gamma_{i,\rho}}$
& $\lambda T$
& $(2\pi)^{\gamma_{i,\rho}/2}(2\alpha)^{-(i-\gamma_{i,\rho})/2}\mu_{i,\rho}^{(d)}\big(\bA^{-1}(B_{\Lambda_m}^{(r)})\big)$\\
 \; (c) Marshall--Olkin business-line claims
& $\Lambda_m^{(r)}$, $i=i_m(\bA)$
& $u^{-\alpha_i^{\star}}$
& $\lambda T$
& $\mu_i^{\star,(d)}\big(\bA^{-1}(B_{\Lambda_m}^{(r)})\big)$\\
\bottomrule
\end{tabular}
\end{table}
\end{landscape}
\FloatBarrier

\end{document}